\documentclass[11pt]{amsart}
\usepackage{amsfonts,amsmath,amssymb}
\usepackage{mathabx}
\usepackage{mathrsfs}
\usepackage[latin1]{inputenc}
\usepackage[usenames,dvipsnames]{pstricks}
\newtheorem{theorem}{Theorem}[section]

\newtheorem{corollary}[theorem]{Corollary}

\newtheorem{lemma}[theorem]{Lemma}

\newcommand{\R}{\mathbb{R}}
\newcommand{\C}{\mathcal{C}}

\newcommand{\N}{\mathbb{N}}

\newcommand{\M}{\Omega}

\newcommand{\F}{\am\times\km}

\newcommand{\p}{\partial}
\newcommand{\s}{\mathcal{S}}
\newcommand{\Rn}{\mathbb{R}^{n}}
\newcommand{\am}{\mathcal{A}(M_0)}
\newcommand{\km}{\mathcal{K}(M_1,M_2)}
\newcommand{\G}{\mathcal{L}}
\newcommand{\norm}[1]{\left\Vert#1\right\Vert}
\newcommand{\abs}[1]{\left|#1\right|}
\newcommand{\set}[1]{\left\{#1\right\}}
\newcommand{\para}[1]{\left(#1\right)}

\newcommand{\A}{\mathscr{A}}
\newcommand{\To}{\longrightarrow}
\usepackage{graphicx}

\usepackage{times}
\newcommand{\rpm}{\raisebox{.2ex}{$\scriptstyle\pm$}}

\begin{document}
\title[ Linear Boltzmann Equation]{An inverse problem for the Linear Boltzmann Equation with time-dependent coefficient}


\author{Mourad Bellassoued}

\address{Universit\'e de Tunis El Manar, Ecole Nationale d'Ing\'enieurs de Tunis, ENIT-LAMSIN, B.P. 37, 1002 Tunis, Tunisia}

\email{mourad.bellassoued@enit.utm.tn}


\author{Yosra Boughanja}

\address{Universit\'e de Tunis El Manar, Ecole Nationale d'Ing\'enieurs de Tunis, ENIT-LAMSIN, B.P. 37, 1002 Tunis, Tunisia}

\email{yosra.boughanja@enit.utm.tn}

%
%
%
%

\date{\today}
\begin{abstract}
In this paper, we study the stability in the inverse problem of determining the time dependent absorption 
coefficient appearing in the linear Boltzmann equation, from boundary observations. We prove in dimension $n\geq 2$, that the absorption 
coefficient can be uniquely determined in a precise subset of the domain,
from the albedo operator. We derive  a logarithm type stability estimate in the determination of the absorption  coefficient from the albedo operator, in
a subset of our domain assuming that it is known outside this subset. Moreover, we prove that we can extend this result to the determination of the coefficient in a larger region, and then in the whole domain provided that we have much more data. We prove also an identification result for the scattering coefficient appearing in the linear Boltzmann equation.
\end{abstract}
\maketitle
\tableofcontents

\section{Introduction and main results}
This article is devoted to the  study of the problem of determining the absorption and scattering properties of a bounded, convex medium $\Omega\subset\Rn$, $n\geq 2$  from the observations made at the boundary.  We denote by $\s=\mathbb{S}^{n-1}$ the unit sphere of $\Rn$, $Q:=\s\times\Omega$ and for $T>0$, we denote $Q_T=(0,T)\times Q$. We consider the linear Boltzmann equation  
\begin{equation}
\p_t u+\theta\cdot\nabla u(t,x,\theta)+a(t,x)u=\G_k[u](t,x,\theta)  \quad  \mbox{in}\quad Q_T	 ,\label{11}
\end{equation}
where $a\in L^\infty((0,T)\times\Omega)$ and $\G_k$ is the integral operator with kernel $k(x,\theta,\theta')$ defined by
\begin{equation}\label{1.2}
\G_k[u](t,x,\theta)=\int_{\s} k(x,\theta,\theta') u(t,x,\theta')d\theta'.
\end{equation}
The function $u$ represent the density of particles at $x$ traveling in the direction $\theta$, $a(t,x)$ is the absorption coefficient at $x$ in the time $t$, and $k(x,\theta,\theta')$ is the scattering coefficient (or the collision kernel). Let $\Gamma^\pm$ denote the incoming and outgoing bundles
\begin{equation}
\Gamma^{\rpm}=\set{(x,\theta)\in\Gamma\,:\, \rpm\theta\cdot\nu(x)>0},\quad \Sigma^{\rpm}=[0,T]\times\Gamma^{\rpm},\quad  \Gamma=\p\Omega\times\s,
\end{equation}
where $\nu=\nu(x)$ denotes the unit outward normal to $\p\M$ at $x$. The medium is probed with the given radiation
\begin{equation}
u_{|\Sigma^-}=f^-.
\end{equation}
The exiting radiation $u_{|\Sigma^+}$ is detected thus defining the albedo operator $\A_{a,k}$ that takes the incoming flux $f^-$ to the outgoing flux $u_{|\Sigma^+}$
\begin{equation}\label{1.55}
\A_{a,k}[f^-]=u_{|\Sigma^+}.
\end{equation}
In this paper, we will study the uniqueness and the stability issues in the inverse problem of determining the time-dependent absorption coefficient  $a$  and the scattering coefficient $k$ from the albedo operator $\A_{a,k}$. We consider three different sets of data and we aim to prove that the absorption coefficient $a$ can be recovered in some specific subdomain of $(0,T)\times\Omega$, by probing it with disturbances generated on the boundary and we obtain also an  identification result for the scattering coefficient $k$.
\smallskip

For general relevant references on theoretical  inverse problems, we refer the reader to \cite{[21], 
 [23], [31], [33], [37]}. For general references and earlier review papers on inverse transport, we refer the reader to e.g. \cite{[1], [27], [34]}.
Inverse transport theory has many applications,  in medical imaging are optical tomography \cite{[2], [31]} and optical molecular imaging \cite{[14]}. Applications in remote sensing in the atmosphere are considered in \cite{[26]}. Inverse transport can also be used efficiently for imaging using high frequency waves propagating in highly heterogeneous media; see e.g. \cite{[5], [6]}.
\smallskip

Most of the paper will be concerned with answering the questions of what may be
reconstructed in $a$ and $k$ from  knowledge of the albedo operator $\A_{a,k}$ and with which stability estimate. This is the inverse transport problem.
\smallskip

The problem of identifying coefficients appearing in the linear Boltzmann equation was treated very well and there are  many works related to this topic. The uniqueness of the reconstruction of the optical parameters from the albedo operator for the stationary Boltzmann equation, both in the case that  the the time-independent absorption coefficient $a$ depend only on position $x$ and the case that $a$ depend also on the direction  $\theta$, was proved by Choulli and Stefanov in \cite{[155]} and Tamasan \cite{[27+]}. As for the stability results was obtained in the stationary Boltzmann equation   
in two or three dimensional case  under smallness assumptions for the absorption and the scattering coefficients  by Romanov \cite{[32], [33]} and  in two-dimensional case under smallness assumptions for the scattering parameter by Stefanov and Uhlmann \cite{[37]}. In three or higher dimensions, stability of the reconstruction results of general  scattering and absorption
coefficients  were established  by Bal and  Jollivet \cite{[4]}. In \cite{[36]} it is shown that the albedo operator determines the pairs of coefficients up to a gauge transformation, and a stability estimate for gauge equivalent classes was proved for the stationary transport equation in \cite{[19]}.
\smallskip

In the case of the dynamic transport equation or linear Boltzmann equation \eqref{11} with time independent coefficients, unique recovery of the pair $(a,k)$  from knowledge of the albedo operator  \eqref{1.55} was shown by Choulli and Stefanov in \cite{[15]}, and stable recovery of the absorption coefficient $a$  was proved by Cipolatti, Motta and Roberty in \cite{[17]} and Cipolatti \cite{[16]}.  This latter result is extended to stable determination of the time-independent pair $(a,k)$ from the albedo map by Bal and Jollivet in \cite{[40]}. However, in this article, we deal with the case where $a$ depend on the spacial variable $x$ and the time $t.$
\smallskip

However, all the mentioned above papers are concerned only with time-independent coefficients. Inspired by the work of  Bellassoued and Ben Aïcha \cite{[8]} and Ben Aïcha \cite{[9]} for the the dissipative wave equation, we prove in this paper stability estimates in the recovery of the  time-dependent absorption  coefficient $a$ appearing in the dynamical Boltzmann equation via different types of measurements and over different subdomain of $  (0,T)\times\Omega $.
\smallskip

The Linear Boltzmann  equation is one of the mathematical models to support the diffuse optical tomography,  here the inverse problem can be set up associated with a map from the input data on one portion of the boundary to the measurement on the other portion, this map is the albedo operator. Depending on the data-acquisition method in the experiments, the measurement can take various forms.
\medskip

Literature dealing with the inverse problem of recovering time-dependent potentials of the transport equation is rather sparse.  In the current work, we show that it is possible to logarithm-stably recover the time-and-space-dependent coefficients of the linear Boltzmann equation (1.1) in a suitable region dependent on the disturbance used to probe the medium. 
\medskip

From a physical view point, our inverse problem consists in determining the time dependent absorption coefficient and the scattering coefficient $k$ in the linear Boltzmann equation (1.1) by probing it with disturbances generated on the boundary and the initial data. 
\medskip

The unique determination of the time-dependent coefficient $a$ in the whole domain $(0,T)\times\Omega$ is fails (see Theorem 1.3) if we assume  that the medium is quiet initially ($u_0=0$) and our data are only the response $\mathscr{A}_{a,k}(f)$ where $f$ denotes the disturbance used to probe the medium. 
\medskip

Further, we proved a log-type stability estimate for the time-dependent absorption coefficient $a$ with respect to the albedo operator  in a subset of $(0,T)\times\Omega$. By assuming that we know $u (T,\cdot,\cdot)$,  with $u_0= 0$, we  extended this result to a larger region and finally, we proved a log-type stability estimate for this problem over the whole domain $(0,T)\times\Omega$ when our data is the response of the medium for all possible initial states.
\subsection{Notations and well-posedness of the linear Boltzmann equation}
We start by examining the well-posedness of the following initial-boundary value problem for the linear Boltzmann equation (\ref{11})
\medskip
\begin{equation}
\left\{
  \begin{array}{ll}
\partial_tu+\theta\cdot\nabla u(t,x,\theta)+a(t,x)u=\G_k[u](t,x,\theta)  \quad & \mbox{in}\quad Q_T:=(0,T)\times Q ,\\      
 u(0,x,\theta)=u_0(x,\theta)& \mbox{in}\quad Q:=\s\times\Omega ,\\
u(t,x,\theta)=f(t,x,\theta) & \mbox{on}\quad \Sigma^-.\\
\end{array}
\right.\label{1}
\end{equation}
In order to well define the albedo operator as a trace operator on the outcoming
boundary, we consider the spaces $\mathscr{L}^{\pm}_p(\Sigma^{\pm}):=L^p(0,T;L^p(\Gamma^{\pm},d\xi))$, where $d\xi=\vert\theta\cdot \nu(x)\vert dx d\theta$, equipped with the norm
$$\Vert f\Vert_{\mathscr{L}_p^{\pm}(\Sigma^\pm)}=\Vert f\Vert_{L^p(0,T;L^p(\Gamma^{\pm},d\xi))}.$$
We introduce the spaces
$$ W_p:=\lbrace u \in  L^p(Q),  \;\theta\cdot \nabla u \in L^p(Q);\;\int_\Gamma \vert \theta\cdot\nu(x)\vert u(x,\theta) \vert^p d\theta dx < \infty  \rbrace.$$
We also consider the operator $A : D(A) \rightarrow L^p(Q)$, defined by $(Au)(x,\theta) = \theta\cdot\nabla u( x,\theta)$,
with $$D(A)=\lbrace u\in W_p;\;u_{|\Gamma^+}=0 \rbrace.$$
Let $\Omega_T=(0,T)\times\Omega$. Given $M_0, M_1, M_2>0$, we consider the set of admissible coefficients $a$ and $k$
\begin{eqnarray*}
\mathcal{A}(M_0)=\lbrace a\in \C^1(\overline{\Omega}_T);\,\Vert a\Vert_{L^\infty(\Omega_T)}\leq M_0\rbrace .
\end{eqnarray*}
\begin{eqnarray*}
\mathcal{K}(M_1,M_2)=\lbrace k\in L^\infty(\Omega,L^2(\s\times\s));\, \int_\s\vert k(x,\theta,\theta')\vert d\theta'\leq M_1,\, \int_\s\vert k(x,\theta,\theta')\vert d\theta\leq M_2\rbrace .
\end{eqnarray*}
Before stating the main results, we recall the following Lemma on the unique existence of a weak solution to problem  \eqref{1}. The proof is based on \cite{[18]} which was completed in Appendix A.
\begin{lemma}\label{Th0}
Let $T>0$, $p\geq 1$ are given, $(a,k)\in \am\times\km$ and $u_0\in L^\infty(Q)$. For $f\in {\mathscr{L}_p^{-}(\Sigma^-)}$ the unique solution $u$ of (\ref{1}) satisfies
\begin{equation*}
u\in\mathcal{C}^1([0,T];L^p(Q))\cap\mathcal{C}([0,T];W_p).
\end{equation*}
Moreover, there is a constant $C>0$ such that
\begin{equation}
\norm{u(t,\cdot,\cdot)}_{L^p(Q)}+\Vert u \Vert _{\mathscr{L}^+_p(\Sigma^+)}\leq C(\norm{u_0}_{L^p(Q)}+ \Vert f \Vert _{\mathscr{L}^-_p(\Sigma^-)}),\quad \forall\,t\in (0,T).\label{aa}
\end{equation}
\end{lemma}
We consider now the following initial boundary value problem for the linear Boltzmann equation  with source term $v\in L^p(Q_T)$:
\begin{equation}\label{1.4}
\left\{
\begin{array}{llll}
\partial_t\psi+\theta\cdot\nabla\psi+a(t,x)\psi=\G_k[\psi]+v (t,x,\theta) & \textrm{in }\,\, Q_T,\cr
\psi(0,x,\theta)=0& \textrm{in}\,\,Q,\cr
\psi(t,x,\theta)=0 & \textrm{on} \,\, \Sigma^-.
\end{array}
\right.
\end{equation}
\begin{lemma}\label{L.1.1}
Let $T>0$, $(a,k)\in\am\times\km$. Assuming $v\in L^p(Q_T)$, then there exists a unique solution
$\psi$ to \eqref{1.4} such that
 $$
 \psi\in \mathcal{C}^1([0,T];L^p(Q))\cap \mathcal{C}([0,T];D(A)).
 $$ 
Furthermore, there is a constant $C>0$ such that
\begin{equation}\label{1.5}
\norm{\psi(t,\cdot,\cdot)}_{L^p(Q)}\leq C\norm{v}_{L^p(Q_T)},\quad \forall\,t\in (0,T).
\end{equation}
\end{lemma}

\subsection{Main results}
Before stating our main results we give the following notations.
Let  $r>0$ be such that $T>2r.$ Without loss of generality, we assume that   $\Omega \subseteq B(0, r/2) =\left\lbrace  x \in \Rn, \vert x\vert <r/2\right\rbrace$. We set $X_r=(0,T)\times B(0, r/2) $ 
and we consider the annular region around the domain $\Omega$,
$$\mathcal{B}_r=\left\lbrace x\in \Rn, \; \frac{r}{2}<\vert x \vert < T-\frac{r}{2}\right\rbrace.   $$
 Setting the {\it{forward}} and {\it{backward}} light cone:
\begin{equation}\label{c+}
\mathcal{C}_r^+=\left\lbrace  (t,x)\in X_r,\; \vert x\vert <t-\frac{r}{2},\;t>\frac{r}{2}\right\rbrace,
\end{equation}
\begin{equation}\label{c-}
\mathcal{C}_r^-=\left\lbrace  (t,x)\in X_r,\; \vert x\vert <T-\frac{r}{2}-t,\;T-\frac{r}{2}>t\right\rbrace.
\end{equation}
Let the {\it{cloaking}} region
\begin{equation}\label{cr}
\mathcal{C}_r=\left\lbrace (t,x)\in X_r, \; \vert x\vert \leq\frac{r}{2}-t,\;0\leq t\leq\frac{r}{2} \right\rbrace .
\end{equation}
Finally, we denote 
$$X_{r,\sharp}=\Omega_T\cap \mathcal{C}_r^+,\quad \quad\text{and}\quad X_{r,*}= \Omega_T\cap \mathcal{C}_r^+\cap\mathcal{C}_r^-.$$
It is clear that, $X_{r,*}\subset X_{r,\sharp}\subset\Omega_T .$ 
\begin{figure}[h]
  \begin{center}
\includegraphics[scale=0.52]{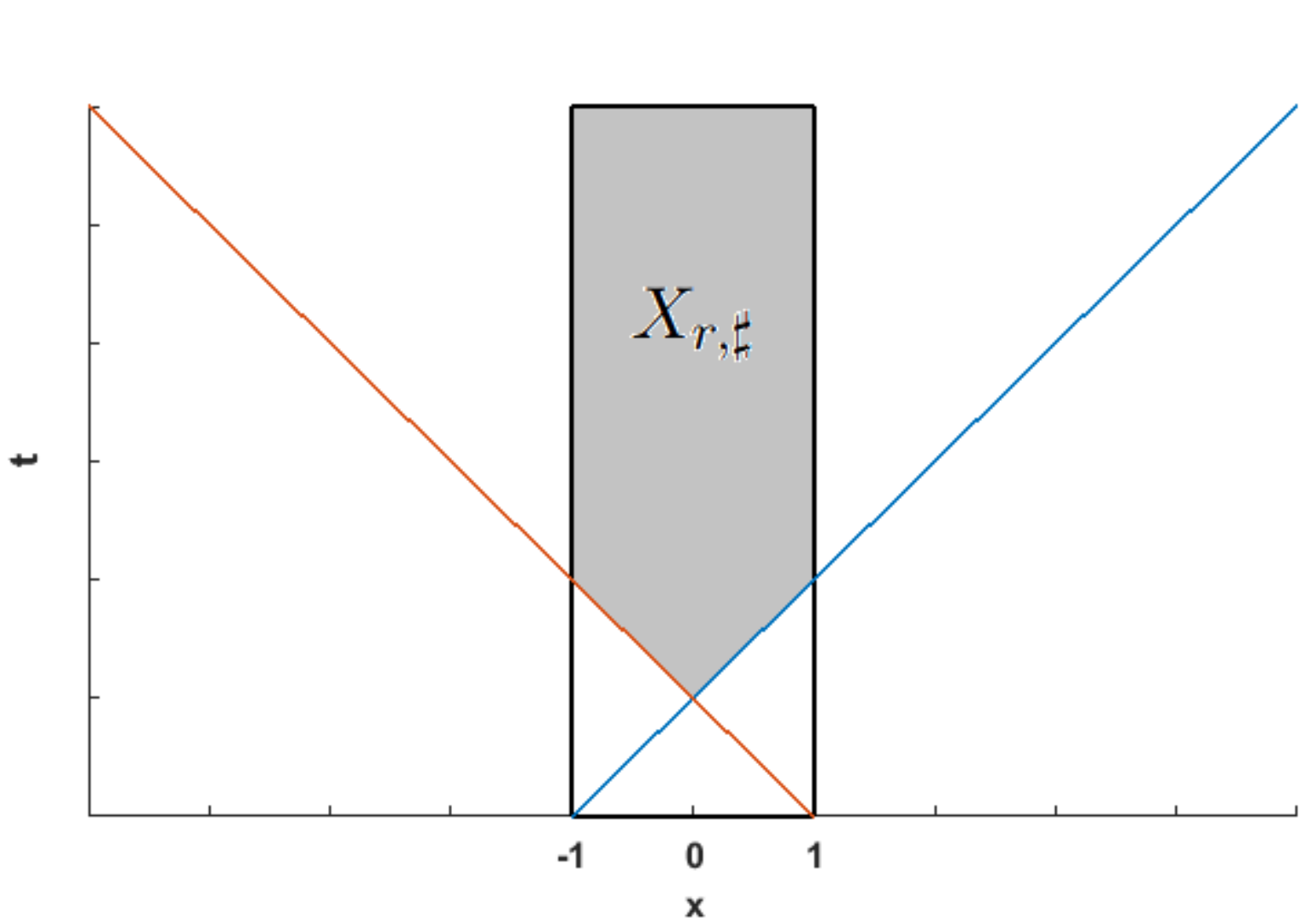} \;
\includegraphics[scale=0.47]{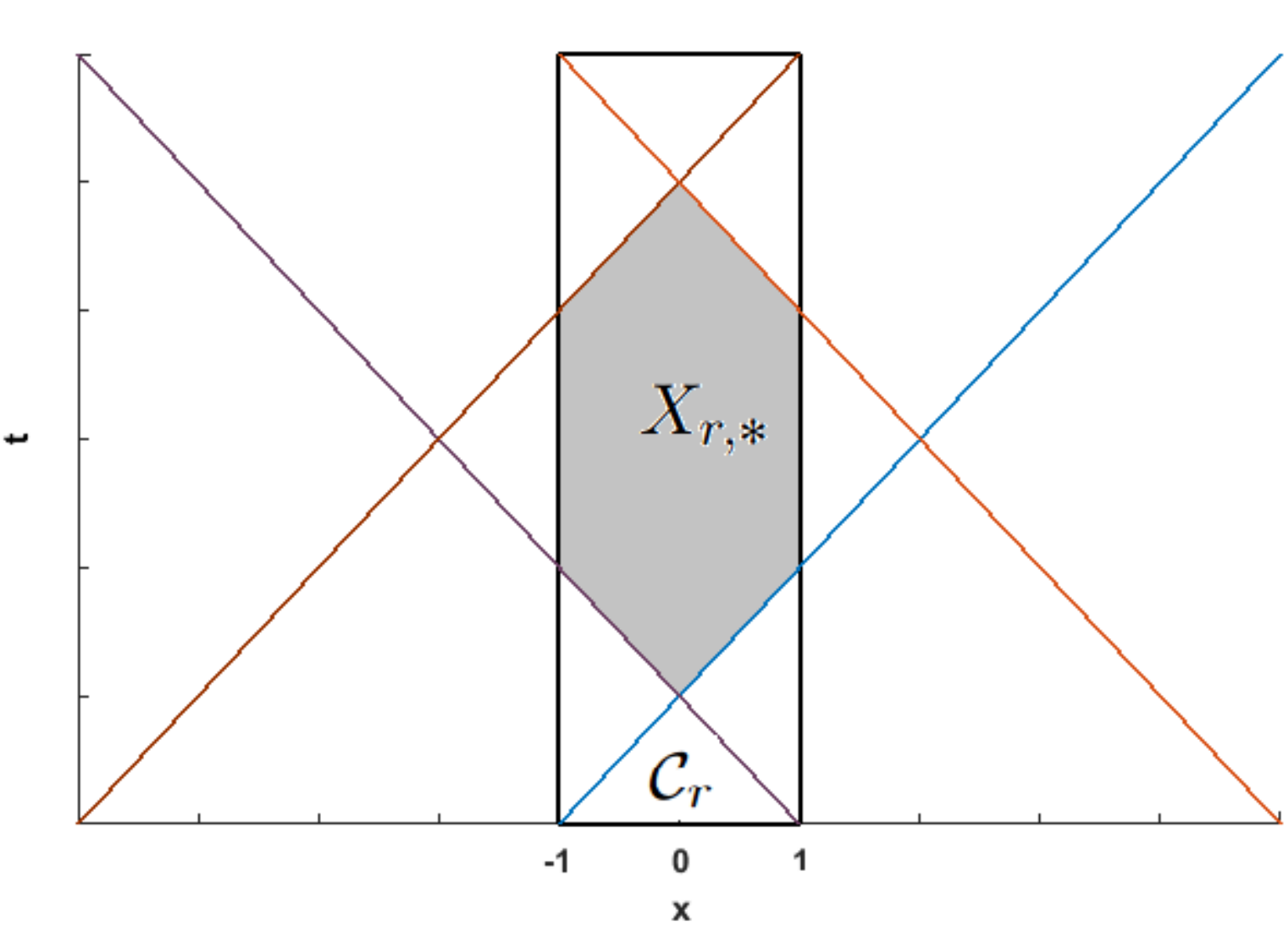} 
\caption{Particular case $\Omega=[-1,1]$}
\end{center}
\end{figure} 
\smallskip

Firstly, assuming that the initial condition $u_0=0$ and our known data will be given only by boundary measurements. That will lead us to consider the albedo operator $\A_{a,k}$, that maps the incoming flux on the boundary $\Sigma^-$ into the outgoing one, defined as follows
$$\begin{array}{llll}
\A_{a,k}:&{\mathscr{L}^-_p(\Sigma^-)}&\longrightarrow  &{\mathscr{L}^+_p(\Sigma^+)}\\
\ &\ \ \ \ f&\longmapsto &\A_{a,k}[f]= u_{|\Sigma^+}.
\end{array}$$ 
From \eqref{aa} the albedo operator map $\A_{a,k}$ is continuous from  ${\mathscr{L}^-_p(\Sigma^-)}$ to ${\mathscr{L}^+_p(\Sigma^+)}$. We denote by $\Vert\A_{a,k}\Vert_p$  its norm.  We can prove that it is hopeless to uniquely determine the absorption coefficient $a$ in the case where this coefficient is supported in the {\it{cloaking}} region $\mathcal{C}_r$.
\smallskip\\
The first result of this paper can be stated as follows:
\begin{theorem} \label{TT1}(Non-uniqueness)
For any $(a,k) \in \mathcal{A}(M_0)\times\km$, such that $\mbox{Supp}(a)\subset \mathcal{C}_r$, we have
$\A_{a,k}=\A_{0,k}.$
\end{theorem}
Let us now introduce the admissible set of the absorption coefficient $a$ in $ X_{r,*}$. Given $a_0 \in \C^1(\overline{\Omega}_T)$, we set
\begin{eqnarray*}
\mathcal{A}^*(M_0) = \lbrace a \in\am; \;a =a_0 \;\text{in} \;{\overline{\Omega}_T}\setminus X_{r,*}\rbrace.
\end{eqnarray*}
\begin{theorem} \label{20}
Let $T>2 r$. There exist $C>0$ and $ \mu,m\in (0, 1)$, such that if $\Vert\A_{a_1,k_1}-\A_{a_2,k_2} \Vert_1\leq m,$ we have 
\begin{equation}\label{222}
\begin{split}
\Vert a_1-a_2\Vert_{H^{-1}( X_{r,*})}\leq C\left( \Vert\A_{a_1,k_1}-\A_{a_2,k_2} \Vert_1^{\mu/2}+\vert\log\Vert\A_{a_1,k_1}-\A_{a_2,k_2} \Vert_1\vert^{-1}\right), 
\end{split}
\end{equation}
for any $(a_j,k_j) \in \mathcal{A}^*(M_0)\times\km$, $j=1,2$.\\
Moreover, if we assume that  $a_j\in H^{s+1}(\Omega_T), $ $s>\frac{n-1}{2}$, with $\Vert a_j\Vert_{ H^{s+1}(\Omega_T)}\leq M,$
$j = 1, 2,$  for some constant $M > 0$, then there exist $C > 0$ and $\mu_1\in (0,1)$ such that
\begin{equation}\label{232}
\begin{split}
\Vert a_1-a_2\Vert_{L^\infty( X_{r,*})}\leq C\left( \Vert\A_{a_1,k_1}-\A_{a_2,k_2} \Vert_1^{\mu/2}+\vert\log\Vert\A_{a_1,k_1}-\A_{a_2,k_2} \Vert_1\vert^{-1}\right)^{\mu_1}, 
\end{split}
\end{equation}
where $C$ depends only on $\Omega, T,n$ and $M_i$, $i=0,1,2$.
\end{theorem}
The above result claims stable determination of the time-dependent absorption coefficient $a$ from the albedo operator  $\A_{a,k}$, in $X _{r,*}\subset \Omega_T$, provided  $a$ is known outside $X_{r,*}$.\\
As an immediate consequence of Theorem \ref{20}, we have:
\begin{corollary}\label{c}
Let $T > 2r$ and $(a_j,k_j)\in \mathcal{A}^*(M_0)\times\km$, $j = 1, 2$. Then, we
have
$$\A_{a_1,k_1}=\A_{a_2,k_2}\quad\text{implies}\quad a_1=a_2 \;\text{on} \; X_{r,*}.$$
\end{corollary}
In order to extend the above results to a larger region $X_{r,\sharp}\supset X_{r,*}$, we require more information about the solution $u$ of the Boltzmann  equation (\ref{1}). So, in this case we will add the final data of the solution $u$. This leads to defining the following boundary operator:

$$\begin{array}{llll}
\tilde{\A}_{a,k}:&{\mathscr{L}^-_p(\Sigma^-)}&\longrightarrow &\mathscr{K}^+_p:={\mathscr{L}^+_p(\Sigma^+)}\times L^p(Q)\\
\ &\ \ \ \ f&\longmapsto &\tilde{\A}_{a,k}[f]= ( u_{|\Sigma^+},\, u(T,\cdot,\cdot)),
\end{array}$$ 
where $u$ is solution of \eqref{1},  with  $u_0=0$. It follows from \eqref{aa}, with  $u_0=0,$ that the operator $\tilde{\A}_{a,k}$ is continuous from  ${\mathscr{L}^-_p(\Sigma^-)}$ to $\mathscr{K}^+_p$. Given $a_0\in \C^1(\overline{\Omega}_T)$, we set
\begin{eqnarray*}
\mathcal{A}^\sharp(M_0) = \lbrace a \in\am; \;a =a_0 \;\text{in} \;\overline{\Omega}_T \setminus X_{r,\sharp}\rbrace.
\end{eqnarray*}
\begin{theorem} \label{Theorem2}
Let $T>2 r$. There exist $C>0$ and  $ \mu,m\in (0, 1)$, such that if $\Vert\tilde{\A}_{a_1,k_1}-\tilde{\A}_{a_2,k_2} \Vert_1\leq m,$ we have 
\begin{equation}\label{222-}
\begin{split}
\Vert a_1-a_2\Vert_{H^{-1}( X_{r,\sharp})}\leq C\left( \Vert\tilde{\A}_{a_1,k_1}-\tilde{\A}_{a_2,k_2} \Vert_1^{\mu/2}+\vert\log\Vert\tilde{\A}_{a_1,k_1}-\tilde{\A}_{a_2,k_2} \Vert_1\vert^{-1}\right), 
\end{split}
\end{equation}
for any  $(a_j,k_j) \in \mathcal{A}^\sharp(M_0)\times \km$, $j=1,2$.\\
Moreover, if we assume that $a_j\in H^{s+1}(\Omega_T), $ $s>\frac{n-1}{2}$, with  $\Vert a_j\Vert_{ H^{s+1}(\Omega_T)}\leq M,$
$j = 1, 2,$ for some constant $M > 0$, then there exist  $C > 0$ and $\mu_1\in (0,1)$ such that
\begin{equation}\label{232-}
\begin{split}
\Vert a_1-a_2\Vert_{L^\infty( X_{r,\sharp})}\leq C\left( \Vert\tilde{\A}_{a_1,k_1}-\tilde{\A}_{a_2,k_2} \Vert_1^{\mu/2}+\vert\log\Vert\tilde{\A}_{a_1,k_1}-\tilde{\A}_{a_2,k_2} \Vert_1\vert^{-1}\right)^{\mu_1}, 
\end{split}
\end{equation}
where $C$ depends only on $\Omega, T,n$ and $M_i$, $i=0,1,2$.
\end{theorem}
In the previous two case, we can see that we can not recover the unknown absorption coefficient $a$ over the whole domain, since the initial data $u_0$ is zero. However, we shall prove that this is no longer the case by considering all possible initial data.
We define the following space $\mathscr{K}^-_p:=\mathscr{L}^-_p(\Sigma^-)\times L^p(Q)$. In this case we will consider observations given by the following operator:
$$\begin{array}{llll}
\mathcal{I}_{a,k}:  & \ \ \ \ \quad \mathscr{K}^-_p&\longrightarrow &\mathscr{K}^+_p\\
 & g=(f,u_0)&\longmapsto &\mathcal{I}_{a,k}(g)= (u_{|\Sigma^+},u(T,\cdot,\cdot)),
\end{array}$$ 
where $u$ is the unique  solution of \eqref{1}. It follows from \eqref{aa} that the operator  $\mathcal{I}_{a,k}$ is continuous from  $\mathscr{K}^-_p$ to $\mathscr{K}^+_p$. Having said that, we are now in position to state the last main result.
\begin{theorem} \label{Theorem3}
Let $T>2r$. There exist $C>0$ and $ \mu,m\in (0, 1)$, such that if  $\Vert\mathcal{I}_{a_1,k_1}-\mathcal{I}_{a_2,k_2} \Vert_1\leq m,$ we have
\begin{equation}\label{2223}
\begin{split}
\Vert a_1-a_2\Vert_{H^{-1}( \Omega_T)}\leq C\left( \Vert\mathcal{I}_{a_1,k_1}-\mathcal{I}_{a_2,k_2} \Vert_1^{\mu/2}+\vert\log\Vert\mathcal{I}_{a_1,k_1}-\mathcal{I}_{a_2,k_2} \Vert_1\vert^{-1}\right), 
\end{split}
\end{equation}
for any $(a_j ,k_j)\in\am\times\km$, $j=1,2.$\\
Moreover, if we assume that $a_j\in H^{s+1}(\Omega_T), $ $s>\frac{n-1}{2}$, with $\Vert a_j\Vert_{ H^{s+1}(\Omega_T)}\leq M,$ 
$j = 1, 2,$ for some constant $M > 0$, then there exist  $C> 0$ and $\mu_1\in (0,1)$ such that
\begin{equation}\label{2323}
\begin{split}
\Vert a_1-a_2\Vert_{L^\infty( \Omega_T)}\leq C\left( \Vert\mathcal{I}_{a_1,k_1}-\mathcal{I}_{a_2,k_2} \Vert_1^{\mu/2}+\vert\log\Vert\mathcal{I}_{a_1,k_1}-\mathcal{I}_{a_2,k_2} \Vert_1\vert^{-1}\right)^{\mu_1}, 
\end{split}
\end{equation}
where $C$ depends only on $\Omega, T,n$ and $M_i$, $i=0,1,2$.
\end{theorem}
Using the same ideas considered in the proof of Theorem \ref{20} and assuming that the scattering coefficient is of the form  $k(x, \theta, \theta') = \rho(x)\kappa(\theta,\theta')$, where $\kappa(\theta,\theta')$ is assumed known, we obtain the following
 result on the identification of both $a$ in $X_{r,*}$ and $\rho$ in $\Omega.$


\begin{theorem}\label{t3}

Let  $(a_j ,k_j)\in\mathcal{A}^*(M_0)\times\km$, $j=1,2$, such that $k_j(x, \theta, \theta') = \rho_j(x)\kappa(\theta,\theta')$,  with $\rho_j\in L^\infty(\Omega),$ $\kappa \in L^\infty(\s \times \s)$ and $\kappa(\theta,\theta)\neq 0,$   $\forall\theta\in\s$. If $T > 2r$ and $\A_{a_1,k_1}=\A_{a_2,k_2}$, then $a_1 = a_2$ in $X_{r,*}$ and  $\rho_1 = \rho_2$ a.e, in $\M$.

\end{theorem}


\section{Non uniqueness in determining the time-dependent absorption coefficient}
\setcounter{equation}{0}
In this section we aim to show that it is hopeless to recover the time-dependent absorption coefficient $a$ in the {\it{cloaking}} region $\mathcal{C}_r$ in the case where the initial condition is zero.
\subsection{Preliminary}
Let us first introduce the following notations. We define
$$V=\bigcup_{0\leq \tau \leq t'}D(\tau)=\bigcup_{0\leq \tau \leq t'}\left(\mathcal{C}_r\cap\left\lbrace t=\tau \right\rbrace  \right), \;$$
where $0<t'<r/2$ and  $\mathcal{C}_r $ is defined by \eqref{cr}. Moreover, we denote by
$$\mathscr{S}=\p \mathcal{C}_r \cap (\Omega\times ]0,t'[),\quad \text{and}\quad \p V=\mathscr{S}\cup D(t')\cup D(0).$$
We denote $d\sigma$  the surface element of $\mathscr{S}$ and the vector $(\eta,\mu_1,\mu_2,\cdots,\mu_n)\in \R^{n+1}$ is
the outward unit normal vector at $(x, t) \in {\mathscr{S}}$ such that
\begin{equation}\label{44}
\eta=\big(\sum_{j=1}^n \mu_j^2\big)^{1/2}.
\end{equation}
\begin{lemma}\label{12344}
Let  $(a,k)\in\mathcal{A}(M_0)\times\km$. For $f\in \mathscr{L}_p^-(\Sigma^-)$ we denote by $u$ the solution of the linear Boltzmann equation
\begin{equation*}
\left\{
  \begin{array}{ll}
\partial_t u+\theta\cdot\nabla u+a(t,x)u=\G_k[u](t,x,\theta)& \mbox{in}\quad Q_T ,\\      
 u(0,x,\theta)=0& \mbox{in}\quad Q ,\\
u(t,x,\theta)=f(t,x,\theta) & \mbox{on}\quad  \Sigma^-.\\
\end{array}
\right.
\end{equation*}
Then, for any $(t,x)\in \mathcal{C}_r $ and $\theta\in \s$, we have $u(t,x,\theta)=0.$
\end{lemma}
\begin{proof}
We denote by $Pu=\p_t u+\theta\cdot\nabla  u +a(t,x)u-\G_k[u]$. By direct calculations, we easily verify that
\begin{align*}
\int_{\s}&\int_V  2P u(t,x,\theta)u(t,x,\theta) dx dtd\theta \\
&=\int_{\s}\int_V 2 \partial_t u(t,x,\theta) u(t,x,\theta)dx dtd\theta+\int_{\s}\int_V 2 \theta\cdot\nabla u(t,x,\theta) u(t,x,\theta)dx dtd\theta\\&+\int_{\s}\int_V 2a(t,x)\vert u(t,x,\theta) \vert^2dx dt d\theta-\int_{\s}\int_V 2\G_k[u](t,x,\theta)u(t,x,\theta) dxdt d\theta
\\
&=\int_{\s}\int_V \frac{d}{dt}\vert u(t,x,\theta) \vert^2 dxdtd\theta+\int_{\s}\int_V \text{div}(\theta \; \vert u(t,x,\theta)\vert^2)dxdtd\theta\\&+\int_{\s}\int_V 2a(t,x)\vert u(t,x,\theta)\vert^2 dx dtd\theta-\int_{\s}\int_V 2\G_\kappa[u](t,x,\theta)u(t,x,\theta) dxdt d\theta
.\end{align*}
Applying the divergence theorem, one gets
\begin{multline*}
\int_{\s}\int_V 2P u(t,x,\theta)u(t,x,\theta) dx dtd\theta\\
=\int_{\s}\int_{\mathscr{S}} \vert u(t,x,\theta) \vert^2 \eta d\sigma d\theta+\int_{\s}\int_{\mathscr{S}} (\theta \cdot \mu) \vert u(t,x,\theta)\vert^2 d\sigma d\theta\\+\int_{\s}\int_{D(t')}\vert u(t', x,\theta) \vert^2 dxd\theta-\int_{\s}\int_{D(0)}\vert u(0,x,\theta) \vert^2 dxd\theta\\+\int_{\s}\int_V 2 a(t,x)\vert u(t,x,\theta)\vert^2 dx dtd\theta-\int_{\s}\int_V 2\G_\kappa[u](t,x,\theta)u(t,x,\theta) dxdt d\theta.
\end{multline*}
By using  (\ref{44}), we can see that
\begin{align*}
\int_{\s}\int_{\mathscr{S}} \vert u(t,x,\theta) \vert^2 \eta d\sigma d\theta+\int_{\s}\int_{\mathscr{S}} (\theta\cdot \mu) \vert u(t,x,\theta)\vert^2    d\sigma d\theta\geq \int_{\s}\int_{\mathscr{S}}\vert u(t,x,\theta)\vert^2   (\eta-\vert\mu\vert) d\sigma d\theta = 0.
\end{align*}
Then, since $u(0,x,\theta)=0,$ we get
\begin{multline*}
\int_{\s}\int_{D(t')}\vert u(t',x,\theta) \vert^2 dxd\theta \leq \int_{\s}\int_V 2P u(t,x,\theta)u(t,x,\theta) dx dtd\theta\cr-\int_{\s}\int_V 2\vert a(t,x)\vert\vert u(t,x,\theta)\vert^2 dx dtd\theta+\int_{\s}\int_V 2\vert \G_k[u](t,x,\theta)\vert \vert u(t,x,\theta)\vert dxdt d\theta.
\end{multline*}
On the other hand, by Cauchy-Schwartz inequality and Fubini's theorem, we can see that
\begin{equation}\label{42}
\int_{\s}\int_V 2\G_k[u](t,x,\theta)u(t,x,\theta) dxdt d\theta
\leq C \int_{\s}\int_V \vert u(t, x,\theta) \vert^2 dx dt d\theta,
\end{equation}
where, the positive constant $C$ is depending on $M_1$ and $M_2$. We can also see that
\begin{equation}\label{422}
\int_{\s}\int_V 2 \vert a(t,x)\vert \vert u(t,x,\theta)\vert^2 dx dtd\theta \leq 2M_0 \int_{\s}\int_V \vert u(t,x,\theta) \vert^2 dx dt d\theta.
\end{equation}
Now, using the fact that $Pu( t,x,\theta) = 0$ for any $( t,x) \in V$ and $\theta\in \s$, yet from (\ref{42}) and (\ref{422}) we get
$$\int_{\s}\int_{D(t')}\vert u(t',x,\theta) \vert^2 dxd\theta \leq  C \int_0^{t'}\int_{\s}\int_{D(t)}\vert u(t,x,\theta) \vert^2 dx d\theta dt.$$
In view of Gr\"onwall's Lemma we end up deducing that $ u(t',x,\theta) = 0$ for any $(t',x) \in D(t')$, $t' \in (0, r/2)$ and $\theta\in \s$. This completes the proof of the lemma.
\end{proof}

\subsection{Proof of Theorem \ref{TT1}}
Let $(a,k)\in \am\times\km$ such that $\mbox{Supp}(a)\subset \mathcal{C}_r$. Let\\ $f \in {\mathscr{L}^-_p(\Sigma^-)}$ and $u$ satisfy
\begin{equation*}
\left\{
  \begin{array}{ll}
\partial_t u+\theta\cdot\nabla u+a(t,x)u=\G_\kappa[u](t,\theta,x)& \mbox{in}\quad Q_T ,\\      
 u(0,x,\theta)=0& \mbox{in}\quad Q ,\\
u(t,x,\theta)=f(t,x,\theta)& \mbox{on}\quad  \Sigma^-.\\
\end{array}
\right.
\end{equation*}
Since from Lemma \ref{12344}, we have $u(t,x,\theta) = 0$ for $(t,x)\in \mathcal{C}_r$ and $\theta\in \s$ and  using the fact that\\ $\mbox{Supp}(a)\subset \mathcal{C}_r$, we deduce that $u$ solves also the following transport boundary-value problem

\begin{equation*}
\left\{
  \begin{array}{ll}
\partial_t v+\theta\cdot\nabla v=\G_\kappa[v](t,x,\theta)& \mbox{in}\quad Q_T ,\\      
 v(0,x,\theta)=0& \mbox{in}\quad Q ,\\
v(t,x,\theta)=f(t,x,\theta) & \mbox{on}\quad  \Sigma^-.\\
\end{array}
\right.
\end{equation*}
Then, we conclude that $\A_{a,k}(f) = \A_{0,k}(f)$ for all $f \in  {\mathscr{L}^-_p(\Sigma^-)}$. Then the proof of Theorem \ref{TT1} is completed. 
\section{Geometrical optics solutions for the Boltzmann equation}\label{Sec2}
\setcounter{equation}{0}
The present section is devoted to the construction of special  solutions called geometrical optics solutions (G.O-solutions) for the linear Boltzmann  equation (\ref{1}), which we will use them to  proof of our main results.
Let $\phi^\pm \in \mathcal{C}^\infty_0(\mathcal{B}_r,\mathcal{C}(\s))$.  Then, for any $\theta\in\s$, we have
\begin{equation}\label{3.2}
\textrm{Supp} \, \phi^+(\cdot,\theta)\cap\Omega=\emptyset.
\end{equation}
Since $T>2r$, we have also 
\begin{equation}
(\textrm{Supp}\, \phi^-(\cdot,\theta) \pm T\theta)\cap\Omega=\emptyset.
\end{equation}
We set
$$
\varphi^\pm (t,x,\theta)=\phi^\pm (x-t\theta,\theta),\quad (t,x,\theta)\in\R\times\Rn\times\s.
$$
Then $\varphi^\pm (t,x,\theta)$ solves the following transport equation 
\begin{equation}\label{3.5}
\para{\p_t+\theta\cdot\nabla}\varphi(t,x,\theta)=0, \quad (t,x,\theta)\in\R\times\Rn\times\s.
\end{equation}
Moreover, for $\lambda>0$, setting
\begin{equation}\label{2.5}
\varphi_{\lambda}^\pm (t,x,\theta)=\varphi^\pm (t,x,\theta)e^{\pm i\lambda(t-x\cdot\theta)},\quad (t,x,\theta)\in\R\times\Rn\times\s.
\end{equation}
Finally, for $a\in \am$,  we set
$$
b_a(t,x,\theta)=\exp\para{-\int_0^t a(s,x-(t-s)\theta)ds},\quad (t,x,\theta)\in \R\times\Rn\times\s,
$$
where we have extended $a$ by $0$ outside $\overline{\Omega}_T$. We have
\begin{eqnarray*}
\theta \cdot \nabla b_a(t,x,\theta) &=&-b_a\int_0^t\theta\cdot\nabla a(s,x-(t-s)\theta )\,ds.
\end{eqnarray*}
This implies  in particular that 
\begin{eqnarray*}
\p_t  b_a(t,x,\theta) &=&-a(t,x)b_a(t,x,\theta)-\left( \int_0^t\frac{d}{dt}a(s,x-(t-s)\theta)ds\right) b_a(t,x,\theta)\\
&=&-a(t,x)b_a(t,x,\theta)+\left( \int_0^t \theta\cdot\nabla a(s,x-(t-s)\theta)ds\right)  b_a(t,x,\theta)
.\end{eqnarray*}
Therefore $b_a$ satisfies the following transport equation
\begin{equation}\label{b}
(\p_t+\theta\cdot\nabla+a(t,x))b_a(t,x,\theta)=0, \quad (t,x,\theta)\in\R\times\Rn\times\s.
\end{equation}

\begin{lemma}\label{L2.1}
There exists $C>0$ such that for any $p\geq1$, we have
\begin{equation}\label{2.7}
\norm{\G_k[\varphi_\lambda^\pm b_a](t,\cdot,\cdot)}_{L^p(Q)}\leq C\norm{\phi^\pm}_{L^p(\Rn\times\s)},\quad \forall t\in (0,T), 
\end{equation}
for any $(a,k)\in\am\times\km .$ 
Moreover, we have 
\begin{equation}\label{3005}
\lim_{\lambda\to\infty}\norm{\G_k[\varphi_\lambda^\pm b_a]}_{L^2(Q_T)}=0.
\end{equation}
where $C$ depends only on $\Omega,$ $M_i$, $i=0,1,2$. Here we recall that $\varphi^\pm_\lambda$ are given by (\ref{2.5}) and $\G_k$ is defined by (\ref{1.2}).
\end{lemma}
\begin{proof}
We have
\begin{equation}
\G_k[\varphi_\lambda^\pm b_a](t,x,\theta)=e^{\pm i\lambda t}\Phi^\pm_\lambda(t,x,\theta),
\end{equation}
where $\Phi^\pm_\lambda$ are given by
\begin{equation}
\Phi^\pm_\lambda(t,x,\theta)=\int_\s k(x,\theta,\theta')\varphi^\pm (t,x,\theta')e^{\mp i\lambda x\cdot\theta'}b_a(t,x,\theta')d\theta'.
\end{equation}
By  H\"{o}lder's inequality,  we get
\begin{align}
\abs{\Phi^\pm_\lambda(t,x,\theta)}^p &\leq M_1^{p/{p'}}\int_\s \abs{k(x,\theta,\theta')}\abs{\varphi^\pm (t,x,\theta')b_a(t,x,\theta')}^pd\theta'\cr
&\leq C^p M_1^{p/{p'}}\int_\s \abs{k(x,\theta,\theta')}\abs{\varphi^\pm (t,x,\theta')}^p d\theta',
\end{align}
applying Fubini's Theorem, we deduce that
\begin{align}
\norm{\Phi^\pm_\lambda(t,\cdot,\cdot)}^p_{L^p(Q)} &\leq C^p M_1^{p/{p'}}M_2\int_\s\int_\Omega\abs{\varphi^\pm(t,x,\theta')}^p dx d\theta'\cr
&\leq C^p\int_\s\int_{\Rn}\abs{\phi^\pm(x,\theta')}^p dxd\theta'.
\end{align}
This complete the proof of (\ref{2.7}). Now we write $\Phi^\pm_\lambda$ as
\begin{equation}
\Phi^\pm_\lambda(t,x,\theta)=\int_\s \tilde{k}_{t,x}(\theta,\theta')e^{\mp i\lambda x\cdot\theta'}d\theta',
\end{equation}
where the kernel $\tilde{k}_{t,x}$ is given by
\begin{equation}
\tilde{k}_{t,x}(\theta,\theta')=k(x,\theta,\theta')\varphi^\pm (t,x,\theta')b_a(t,x,\theta'),
\end{equation}
since 
$$
\int_\s\int_\s\abs{\tilde{k}_{t,x}(\theta,\theta')}^2d\theta\,d\theta'<+\infty,\quad \textrm{a.e.,}\,\, (t,x)\in \Omega_T,$$
then the integral operator with kernel $\tilde{k}_{t,x}$ is compact in $L^2(\s)$. Using the fact that $\theta\longmapsto e^{\mp i\lambda x\cdot\theta}$ converge weakly to zero in $L^2(\s)$ when $\lambda\to\infty$, we get
\begin{equation}
\lim_{\lambda\to\infty}\norm{\Phi^\pm _\lambda(t,x,\cdot)}_{L^2(\s)}=0, \quad \textrm{a.e.,}\,\, (t,x)\in \Omega_T.
\end{equation}
The Lebesgue Theorem implies that
\begin{equation}\label{4.55}
\lim_{\lambda\to\infty}\norm{\Phi^\pm_\lambda}_{L^2(Q_T)}=0=\lim_{\lambda\to\infty}\norm{\G_k[\varphi_\lambda^\pm b_a]}_{L^2(Q_T)}.
\end{equation}
This completes the proof.
\end{proof}
\begin{lemma}\label{LC1}
Let $p\geq 1$ and $(a,k)\in \am\times\km$. For any $\phi^+ \in \C^\infty_0(\mathcal{B}_r, \C(\s))$ and any $\lambda>0$ the following system
\begin{equation}\label{1.11}
\left\{
\begin{array}{lll}
\partial_tu+\theta\cdot\nabla u+a(t,x)u=\G_k[u](t,x,\theta)  & \,\,(t,x,\theta)\in\,Q_T, \cr
 u(0,x,\theta)=0 &\,\, (x,\theta)\in\, Q,
\end{array}
\right.
\end{equation}
has a solution of the form
\begin{equation}\label{e4}
u_\lambda^+(t,x,\theta)=\varphi_\lambda^+(t,x,\theta)b_{a}(t,x,\theta)+\psi^+_\lambda(t,x,\theta),
\end{equation}
satisfying
$$
u_\lambda^+\in \mathcal{C}([0,T];W_p(Q))\cap \mathcal{C}^1([0,T];L^p(Q))
,$$
where  $\psi^+_\lambda$ satisfies
\begin{equation}\label{30.26}
\left\{
\begin{array}{llll}
\partial_t\psi+\theta\cdot\nabla\psi+a(t,x)\psi=\G_k[\psi]+\G_k[\varphi_\lambda^+b_a](t,x,\theta)
& \textrm{in }\,\, Q_T,\cr
\psi(0,x,\theta)=0& \textrm{in}\,\,Q,\cr
\psi(t,x,\theta)=0 & \textrm{on} \,\, \Sigma^-,
\end{array}
\right.
\end{equation}
and, there is a constant $C>0$ such that
\begin{equation}\label{4.6}
\norm{\psi^+_\lambda}_{\mathcal{C}([0,T];L^p(Q))}\leq C\norm{\phi^+}_{L^p(\Rn\times\s)}.
\end{equation}
Here $C$ is a constant depending  on $\Omega$, $T$ and $M_i$, $i=0,1,2$.\\
Moreover, we have
\begin{equation}\label{4.66}
\lim_{\lambda\to\infty}\norm{\psi^+_\lambda}_{L^2(Q_T)}=0.
\end{equation}
\end{lemma}
\begin{proof}
Let $\psi^+_\lambda(t,x,\theta)$ solves the following 
homogeneous boundary value problem
\begin{equation}\label{3-.26}
\left\{
\begin{array}{llll}
\partial_t\psi+\theta\cdot\nabla\psi+a(t,x)\psi=\G_k[\psi]+v_\lambda (t,x,\theta)
& \textrm{in }\,\, Q_T,\cr
\psi(0,x,\theta)=0& \textrm{in}\,\,Q,\cr
\psi(t,x,\theta)=0 & \textrm{on} \,\, \Sigma^-,
\end{array}
\right.
\end{equation}
where the source term $v_\lambda$ is given by 
\begin{equation}\label{3.27}
v_\lambda(t,x,\theta)=-\para{\partial_t+\theta\cdot\nabla+a(t,x)}\para{\varphi_\lambda^+(t,x,\theta)b_a(t,x,\theta)}+\G_k[\varphi_\lambda^+b_a](t,x,\theta).
\end{equation}
To prove the Lemma it would be enough to show that $\psi_\lambda^+$ satisfies the estimates (\ref{4.6}) and (\ref{4.66}).\\
By a simple computation and using (\ref{3.5}) and (\ref{b}), we have
\begin{align}\label{3.28}
v_\lambda(t,x,\theta)&=\G_k[\varphi_\lambda^+b_a](t,x,\theta).
\end{align}
Furthermore, thanks to Lemma \ref{L.1.1} there is a constant $C>0$, such that
\begin{equation}\label{2018}
\norm{\psi^+_\lambda(t,\cdot,\cdot)}_{L^p(Q)}\leq C\norm{v_\lambda}_{L^p(Q_T)},\quad \forall\,t\in(0,T).
\end{equation}
From \eqref{2.7}, (\ref{3005}) and (\ref{2018}) we obtain (\ref{4.6}) and  (\ref{4.66}). This completes the proof  of the Lemma.
\end{proof}
We prove by proceeding as in the proof of Lemma \ref{LC1}, the following Lemma.

\begin{lemma}\label{LC2}
Let $q\geq 1$ and $(a,k)\in \F$. For any $\phi^- \in \C_0^\infty(\mathcal{B}_r;\C(\s)) $  and any $\lambda>0$ the following system 
\begin{equation}\label{1.12}
\left\{
\begin{array}{lll}
\partial_tu+\theta\cdot\nabla u-a(t,x)u=-\G_k^*[u](t,x,\theta)  &\,\,(t,x,\theta)\in\,Q_T, \cr
u(T,x,\theta)=0 & \,\,(x,\theta)\in\,Q,
 \end{array}
 \right.
\end{equation}
where
$$ \G_k^*[u](t,x,\theta)=\int_{\s} k(x,\theta',\theta)u(t,x,\theta')d\theta',$$
has $a$ solution of the form
$$
u_\lambda^- (t,x,\theta)=\varphi_\lambda^-(t,x,\theta)b_{-a}(t,x,\theta)+\psi^-_\lambda(t,x,\theta),
$$
satisfying
$$
u_\lambda^-\in \mathcal{C}([0,T];W_q(Q)) \cap \mathcal{C}^1([0,T];L^q(Q))
,$$
where  $\psi^-_\lambda$ satisfies
\begin{equation}\label{3+.26}
\left\{
\begin{array}{llll}
\partial_t\psi+\theta\cdot\nabla\psi-a(t,x)\psi=-\G_k^*[\psi]-\G_k^*[\varphi_\lambda^-b_a](t,x,\theta)
& \textrm{in }\,\, Q_T,\cr
\psi(T,x,\theta)=0& \textrm{in}\,\,Q,\cr
\psi(t,x,\theta)=0 & \textrm{on} \,\, \Sigma^+,
\end{array}
\right.
\end{equation}
and, there is a constant $C>0$ such that
\begin{equation}\label{4.7}
\norm{\psi^-_\lambda}_{\mathcal{C}([0,T];L^{q}(Q))}\leq C\norm{\phi^-}_{L^q(\R^n\times\s)}.
\end{equation}
Here $C$ is a constant depending only on $\Omega$, $T$ and $M_i$, $i=0,1,2$.\\
Moreover, we have
\begin{equation}\label{4.77}
\lim_{\lambda\to\infty}\norm{\psi^-_\lambda}_{L^2(Q_T)}=0.
\end{equation}
\end{lemma}

\section{Recovering of absorption coefficients}

\setcounter{equation}{0}

In this section we prove a log-type stability estimate for $a$ appearing in the initial boundary
value problem (\ref{1}) with $u_0 = 0$. This is by means of the geometric optics solutions
introduced in section 2 and the light-ray transform. We start by considering geometric optics solutions of the form (\ref{e4}). We only assume that  Supp
$\phi^\pm (\cdot,\theta)\subset \mathcal{B}_r$, in such a way we have
$$
\mbox{Supp}(\phi^+)(\cdot,\theta)\cap \Omega=\emptyset\,\quad\mbox{and}\,\,\,(\mbox{Supp}(\phi^-)(\cdot,\theta) \pm T\theta)\cap \Omega=\emptyset,\quad \forall\,\theta\in \s.
$$


\subsection{Stable determination of absorption coefficient from boundary measurements}

The present section is devoted to the proof of Theorem \ref{20}. Our goal here is to show that the time dependent coefficient $a$ depends stably on the albedo operator  $\A_{a,k}$. In the rest of this section, we extend $a$ in $\R^{n+1}$ by $a =a_2-a_1$ in $\overline{\Omega}_T$ and $a =0$ on $\R^{n+1}\diagdown \overline{\Omega}_T$. We start by collecting a preliminary estimate which is needed to prove the main statement of this section.
\subsubsection{Preliminary estimate}
The main purpose of this section is to give a preliminary estimate, which relates the difference of the absorption coefficients to albedo operator. Let $(a_j,k_j)\in\mathcal{A}^*(M_0)\times\km$, $j=1,2,$ . We set
$$
a=a_{2}-a_{1},\,\mbox{and}\,\,b_a(t,x,\theta)=b_{a_2}b_{-a_1}(t,x,\theta)=\exp\Big(-\int_{0}^{t}a(s,x-(t-s)\theta)\,ds\Big),
$$
where we have extend $a$ by $0$ outside $\overline{\Omega}_T$.\\
Here, we recall the definition of $b_{-a_1}$ and $b_{a_2}$
$$b_{-a_1}(x,t)=\exp\Big( \int_{0}^{t}a_{1}(s,x-(t-s)\theta)\,ds \Big),\quad 
b_{a_2}(t,x,\theta)=\exp\Big(-\int_{0}^{t}a_{2}(s,x-(t-s)\theta)\,ds\Big).$$
We also need to define the Poisson kernel of $B(0,1)\subset \R^n$, i.e.,
$$
P(\omega,\theta)=\frac{1-\abs{\omega}^2}{\alpha_n\abs{\omega-\theta}^n},\quad \omega,\, \theta\in\s
.$$
For $0<h<1$, we define $\varrho_h :\s\times \s\to \R$ as
\begin{align}\label{4.16}
\varrho_h(\omega,\theta)=P(h\omega,\theta).
\end{align}
We have the following Lemma (see Theorem 2.46 of \cite{[F]}).
\begin{lemma}\label{L.4.3}
Let $\varrho_h$ given by (\ref{4.16}), $h \in (0,1)$. Then we have the following properties:
\begin{equation}\label{4.17}
\displaystyle 0\leq\varrho_h(\omega,\theta)\leq \frac{2}{\alpha_n(1-h)^{n-1}},\quad \forall\, h\in (0,1),\,\forall\,\theta,\,\omega\in \s,
\end{equation}  
for some positive constant $\alpha_n$.
\begin{equation}\label{4.18}
\displaystyle \int_{\s}\varrho_h(\omega,\theta) d\theta=1,\quad \forall h\in (0,1),\,\forall\,\omega \in \s.
\end{equation} 
For any  $g \in \C(\s),$ we have
\begin{equation}\label{4.19}
 \displaystyle\lim_{h\to 1} \int_{\s}\varrho_h(\omega,\theta)g(\theta)d\theta=g(\omega),\,\forall\,\omega \in \s,
 \end{equation}
where the limit is taken in the topology of $L^p(\s)$, $p\in [1,+\infty[$ and uniformly on   $\s$.
\end{lemma}
The principal estimate of this section can be stated as follows
\begin{lemma}\label{Lemma3.1}
Let $(a_{j},k_j)\in\mathcal{A}^*(M_0)\times\km$, $j=1,\,2$. There exists $C>0$ such that for any  $\phi^\pm\in \C^\infty_0(\mathcal{B}_r,\C(\s)) $, the following estimate holds true
\begin{multline}\label{ll}
\Big|\int_{\R^n\times\s}(\phi^+\phi^-)(y,\theta)
\Big[ \exp\Big( -\int_{0}^{T}a(s,y+s\theta)\,ds \Big)-1 \Big]\,d\theta\, dy  \Big|\cr
\leq C\|\A_{a_{2},k_{2}}-\A_{a_{1},k_{1}}\|_1\|\phi^+\|_{L^{1}(\R^n\times\s)}\|\phi^-\|_{L^{\infty}(\R^n\times\s)}.
\end{multline}
 Here $C$ depends only on
$\Omega$, $T$ and $M_i$, $i=0,1,2$.
\end{lemma}

\begin{proof}
In view of Lemma \ref{LC1}, and using the fact that Supp
$\phi^+(\cdot,\theta)\cap\Omega=\emptyset$, there exists a G.O-solution
$u_\lambda^{+}$ to the equation
$$
\left\{
    \begin{array}{ll}
      \p_{t}u+\theta\cdot\nabla u+a_{2}(t,x)u=\G_{k_2}[u] (t,x,\theta) & \,\,\,(t,x,\theta)\in Q_T, \cr
      u(0,x,\theta)=0 & \,\,\, (t,x)\in Q,
    \end{array}
  \right.$$
in the following form
\begin{equation*}\label{EQ3*.19}
u_\lambda^{+}(t,x,\theta)=\varphi_\lambda^+(t,x,\theta)b_{a_2}(t,x,\theta)+\psi_{\lambda}^{+}(t,x,\theta),
\end{equation*}
corresponding to the coefficients $a_{2}$ and $k_{2}$,  where $\psi_\lambda^{+}(t,x,\theta)$
satisfies (\ref{4.6}) and (\ref{4.66}). Next, let us denote by $f_{\lambda}$
the function
$$f_{\lambda}(t,x,\theta):= u^{+}_\lambda(t,x,\theta)_{|\Sigma^-}=\varphi_\lambda^+(t,x,\theta) b_{a_2}(t,x,\theta)_{|\Sigma^-}.$$
We denote by $u_{1}$ the solution of
$$\left\{
  \begin{array}{ll}
   \p_{t}u_{1}+\theta\cdot\nabla u_{1}+a_{1}(t,x)u_{1}=\G_{k_1}[u_1] & \,\,\,(t,x,\theta)\in Q_T, \cr
    u_{1}(0,x,\theta)=0 & \,\,\,(x,\theta)\in Q ,\cr
    u_{1}(t,x,\theta)=f_{\lambda}(t,x,\theta) &\,\,(t,x,\theta)\in\Sigma^-.
  \end{array}
\right.
$$
Putting $u=u_{1}-u_\lambda^{+}$. Then, $u$ is a solution to the following system
\begin{equation}\label{33.311}
\left\{
  \begin{array}{ll}
     \p_{t}u+\theta\cdot\nabla u+a_{1}(t,x)u=\G_{k_1}[u]+a(t,x)u^+_\lambda-\G_k[u^+_\lambda] & \,\,\,(t,x,\theta)\in Q_T ,\cr
    u(0,x,\theta)=0 & \,\,\,(x,\theta)\in Q, \cr
    u(t,x,\theta)=0 & \,\,\,(t,x,\theta)\in\Sigma^-,
  \end{array}
\right.
\end{equation}
where $a=a_{2}-a_{1}$ and $k=k_{2}-k_{1}$. On the other hand by  Lemma \ref{LC2} and the fact that (Supp $\phi_1(\cdot,\theta)\pm
 T\theta)\cap\Omega=\emptyset$, we can find of a solution $u^{-}_\lambda$ to the backward problem 
$$
\left\{
  \begin{array}{ll}
 \p_{t}u+\theta\cdot\nabla u-a_{1}(t,x)u=-\G^*_{k_1}[u](t,x,\theta) & \textrm{in}\,\,\,Q_T, \cr
    u(T,x,\theta)=0 & \textrm{in}\,\,\,\, Q,
  \end{array}
\right.
$$
corresponding to the coefficients $a_{1}$ and $k_1$, in the form
\begin{equation}\label{EQ3.21}
u^{-}_\lambda(t,x,\theta)=\varphi_\lambda^-(t,x,\theta)b_{-a_1}(t,x,\theta)+\psi_{\lambda}^{-}(t,x,\theta),
\end{equation}
here the remainder term  $\psi_{\lambda}^{-}(t,x,\theta)$
satisfies the estimate (\ref{4.7}) and (\ref{4.77}). Multiplying the first equation of (\ref{33.311}) by $u^{-}_\lambda$, integrating by parts and using Green's formula, we obtain
\begin{multline}\label{EQ3.222}
\int_{Q_T} a(t,x) u^{+}_\lambda\,u^{-}_\lambda\,d\theta\,dx\,dt-\int_{Q_T}\G_k[u^+_\lambda](t,x,\theta)u^{-}_\lambda (t,x,\theta) d\theta\,\,dx\,dt\cr
=\int_{\Sigma^+}\theta\cdot\nu(x) (\A_{a_{2},k_{2}}-\A_{a_{1},k_{1}}) (f_{\lambda}) u_\lambda^{-}\, d\theta\,dx\,dt.
\end{multline}
Where we have used
$$
\int_{Q_T}\G_{k_1}[u](t,x,\theta)u_\lambda^-(t,x,\theta) d\theta\,dx\,dt=\int_{Q_T} u(t,x,\theta)\G^*_{k_1}[u^-_\lambda](t,x,\theta) d\theta\,dx\,dt.
$$
By replacing $u_\lambda^{+}$ and $u_\lambda^{-}$ by their expressions, we obtain

\begin{align}
\int_{Q_T} a(t,x)u_\lambda^{+}\,u_\lambda^{-} d\theta\,dx\,dt&=\int_{Q_T} a(t,x) \varphi^+_\lambda(t,x,\theta)\varphi^-_\lambda(t,x,\theta)b_a(t,x,\theta)d\theta\,dx\,dt\cr
&+\int_{Q_T} a(t,x)\varphi_\lambda^-(t,x,\theta) b_{-a_1}(t,x,\theta)\psi_{\lambda}^{+}(t,x,\theta)d\theta\,dx\,dt\cr
&+\int_{Q_T} a(t,x)\varphi_\lambda^+(t,x,\theta)b_{a_2}(t,x,\theta)\psi_{\lambda}^{-}(t,x,\theta)d\theta\,dx\,dt\cr
&+\int_{Q_T} a(t,x)\psi_{\lambda}^{-}(t,x,\theta)\psi_{\lambda}^{+}(t,x,\theta)d\theta\,dx\,dt\cr
&=\int_{Q_T} a(t,x) \varphi^+_\lambda(t,x,\theta)\varphi^-_\lambda(t,x,\theta)b_a(t,x,\theta)d\theta\,dx\,dt+\mathcal{I}_{1,\lambda}
,\end{align}
where $\mathcal{I}_{1,\lambda}$ represents the sum of the integrals that contain terms with $\psi_{\lambda}^{+}$ and 	$\psi_{\lambda}^{-}$.\\
Note that there exists $C>0$ such that for any $\lambda\geq 0$, we have
\begin{equation}
|\mathcal{I}_{1,\lambda}|\leq C\left( \norm{\psi_\lambda^+}_{L^2(Q_T)}+\norm{\psi_\lambda^-}_{L^2(Q_T)}\right) .
\end{equation}
Then by \eqref{4.66} and  \eqref{4.77}, we get 
\begin{equation}\label{i1}
\lim_{\lambda\to \infty}\mathcal{I}_{1,\lambda}=0.
\end{equation}
Moreover, we have
\begin{align*}
\int_{Q_T} \G_k[u^{+}_\lambda] u^{-}_\lambda d\theta\,dx\,dt&=\int_{Q_T} \G_k[\varphi_\lambda^+b_{a_2}](t,x,\theta) \varphi_\lambda^-(t,x,\theta)b_{-a_1}(t,x,\theta)d\theta\,dx\,dt\cr
&+\int_{Q_T} \G_k[\psi_{\lambda}^{+}](t,x,\theta)\varphi_\lambda^-(t,x,\theta) b_{-a_1}(t,x,\theta)d\theta\,dx\,dt\cr
&+\int_{Q_T} \G_k[\varphi_\lambda^+b_{a_2}](t,x,\theta)\psi_{\lambda}^{-}(t,x,\theta)d\theta\,dx\,dt\cr
&+\int_{Q_T} \G_k[\psi_{\lambda}^{+}](t,x,\theta)\psi_{\lambda}^{-}(t,x,\theta)d\theta\,dx\,dt\cr
&=\int_{Q_T} \G_k[\varphi_\lambda^+b_{a_2}](t,x,\theta) \varphi_\lambda^-(t,x,\theta)b_{-a_1}(t,x,\theta)d\theta\,dx\,dt+\mathcal{I}_{2,\lambda},
\end{align*}
where 
\begin{multline*}
\mathcal{I}_{2,\lambda}=\int_{Q_T} \G_k[\psi_{\lambda}^{+}](t,x,\theta)\varphi_\lambda^-(t,x,\theta) b_{-a_1}(t,x,\theta)d\theta\,dx\,dt\cr
+\int_{Q_T} \G_k[\varphi_\lambda^+b_{a_2}](t,x,\theta)\psi_{\lambda}^{-}(t,x,\theta)d\theta\,dx\,dt
+\int_{Q_T} \G_k[\psi_{\lambda}^{+}](t,x,\theta)\psi_{\lambda}^{-}(t,x,\theta)d\theta\,dx\,dt.
\end{multline*}
By \eqref{2.7}, we get
\begin{equation}\label{ii2}
|\mathcal{I}_{2,\lambda}|\leq C\left( \norm{\psi_\lambda^+}_{L^2(Q_T)}+\norm{\psi_\lambda^-}_{L^2(Q_T)}\right) .
\end{equation}
Then by \eqref{4.66} and  \eqref{4.77}, we have 
\begin{equation}\label{i2}
\lim_{\lambda\to \infty}\mathcal{I}_{2,\lambda}=0,
\end{equation}  
and by the Cauchy-Schwartz inequality, we obtain 

%

\begin{equation}
\abs{\int_{Q_T} \G_k[\varphi_\lambda^+b_{a_2}](t,x,\theta) \varphi_\lambda^-(t,x,\theta)b_{-a_1}(t,x,\theta)d\theta\,dx\,dt}\leq C\norm{ \G_k[\varphi_\lambda^+b_{a_2}]}_{L^2(Q_T)}.
\end{equation}
Then by \eqref{3005}, we have
\begin{equation}\label{i3}
\lim_{\lambda\to \infty}\int_{Q_T} \G_k[\varphi_\lambda^+b_{a_2}](t,x,\theta) \varphi_\lambda^-(t,x,\theta)b_{-a_1}(t,x,\theta)d\theta\,dx\,dt=0.
\end{equation} 
In light of (\ref{EQ3.222}), we have
\begin{multline}\label{EQ3.23}
\int_{Q_T} a(t,x) \varphi^+(t,x,\theta)\varphi^-(t,x,\theta)b_a(t,x,\theta)d\theta\,dx\,dt\\=\int_{\Sigma^+}\theta\cdot\nu(x) (\A_{a_{2},k_{2}}-\A_{a_{1},k_{1}}) (f_{\lambda}) u_\lambda^{-}\, d\theta\,dx\,dt-\mathcal{I}_{\lambda},
\end{multline}
where the remainder term $\mathcal{I}_{\lambda}$ is given by
\begin{equation*}
\mathcal{I}_{\lambda}=\mathcal{I}_{1,\lambda}-\mathcal{I}_{2,\lambda}-\int_{Q_T} \G_k[\varphi_\lambda^+b_{a_2}](t,x,\theta) \varphi_\lambda^-(t,x,\theta)b_{-a_1}(t,x,\theta)d\theta\,dx\,dt.
\end{equation*}
Taking the limit as $\lambda\to +\infty$, we get from (\ref{i1}), (\ref{i2}) and (\ref{i3}) 
\begin{equation}\label{EQ3.24}
\lim_{\lambda\to\infty}\mathcal{I}_{\lambda}=0.
\end{equation}
Furthermore on $\Sigma^-$, we have $u_\lambda^{+}=f_{\lambda}$, then we get the following estimate
\begin{equation}\label{EQ3.25}
\Big|\int_{\Sigma^+}(\A_{a_{2},k_{2}}-\A_{a_{1},k_{1}})(f_{\lambda})\,u^{-}_\lambda\,dx\,d\theta\, dt\Big|\leq C
\|\A_{a_{2},k_{2}}-\A_{a_{1},k_{1}}\|_p\,\|f_\lambda\|_{{\mathscr{L}^-_p(\Sigma^-)}}\|u_\lambda^-\|_{\mathscr{L}^+_q(\Sigma^+)}.
\end{equation}
Using the fact that
\begin{equation}\label{inq1}
\|u_\lambda^-\|_{\mathscr{L}^+_q(\Sigma^+)}\leq C \|\phi^-\|_{L^{\infty}(\R^n\times\s)},\quad \lim_{p\to 1}\|f_\lambda\|_{\mathscr{L}^-_p(\Sigma^-)}=\|f_\lambda\|_{\mathscr{L}^-_1(\Sigma^-)}\leq C \|\phi^+\|_{L^{1}(\R^n\times\s)},
\end{equation}
and note that, since $\mathscr{A}_{a_j,k_j}$, $j=1,2$ are continuous from $\mathscr{L}_{p'}^-(\Sigma^-)$ to $\mathscr{L}_{p'}^+(\Sigma^+)$ for any $p'\geq 1$, we get for any $p$, $1<p<2$, by Riesz-Thorin's Theorem, the following interpolation inequality 
$$
\|\A_{a_{2},k_{2}}-\A_{a_{1},k_{1}}\|_p\leq \|\A_{a_{2},k_{2}}-\A_{a_{1},k_{1}}\|_2^{1-\kappa_p}\|\A_{a_{2},k_{2}}-\A_{a_{1},k_{1}}\|_1^{\kappa_p}.
$$
with $\kappa_p=(2-p)/p$. So that
\begin{equation}\label{inq2}
\limsup_{p\to 1}\|\A_{a_{2},k_{2}}-\A_{a_{1},k_{1}}\|_p\leq \|\A_{a_{2},k_{2}}-\A_{a_{1},k_{1}}\|_1.
\end{equation}
We get, by (\ref{EQ3.23}), (\ref{EQ3.24}) and (\ref{EQ3.25}), 
$$
\Big|\int_{Q_T} a(t,x) \varphi^+\varphi^-(t,x,\theta)b_a(t,x,\theta)d\theta\,dx\,dt\Big|\leq C\|\A_{a_{2},k_{2}}-\A_{a_{1},k_{1}}\|_1 \|\phi^+\|_{L^{1}(\R^n\times\s)}\|\phi^-\|_{L^{\infty}(\R^n\times\s)}.
$$
Then, using the fact $a(t,\cdot)=0$ outside $\Omega$ and by   changing of variables $y=x-t\theta$, one gets the following estimation
\begin{multline*}
\Big|\int_{0}^{T}\int_{\R^{n}\times\s} a(t,y+t\theta)\phi^+(y,\theta)\phi^-(y,\theta)\exp\Big(-\int_{0}^{t}a(s,y+s\theta)\,ds   \Big)\,d\theta dy\,dt  \Big|\cr
\leq C\|\A_{a_{2},k_{2}}-\A_{a_{1},k_{1}}\|_1\|\phi^+\|_{L^{1}(\R^n\times\s)}\|\phi^-\|_{L^{\infty}(\R^n\times\s)}.
\end{multline*}
Bearing in mind that
\begin{multline*}
\int_{0}^{T}\int_{\R^{n}\times\s}a(t,y+t\theta)\phi^+(y,\theta)\phi^-(y,\theta) \exp\Big( -\int_{0}^{t}a(s,y+s\theta)\,ds\Big)\, d\theta\, dy\,dt\cr
=-\int_{0}^{T}\int_{\R^{n}\times\s}\phi^+(y,\theta)\phi^-(y,\theta)\frac{d}{dt}\Big[ \exp\Big( -\int_{0}^{t}a(s,y+s\theta)\,ds  \Big)
\Big]\,d\theta\, dy\,dt\cr 
=-\int_{\R^{n}\times\s }(\phi^+\phi^-)(y,\theta)\Big[\exp \Big( -\int_{0}^{T}a(s,y+s\theta)\,ds  \Big)-1\Big]\,d\theta\, dy,
\end{multline*}
we conclude the desired estimate given by
\begin{multline*}
\Big|\int_{\R^{n}\times\s}(\phi^+\phi^-)(y,\theta)\Big[ \exp\Big( -\int_{0}^{T}a(s,y+s\theta)\,ds \Big)-1 \Big]\,d\theta\, dy  \Big|\cr
\leq C\|\A_{a_{2},k_{2}}-\A_{a_{1},k_{1}}\|_1\|\phi^+\|_{L^{1}(\R^n\times\s)}\|\phi^-\|_{L^{\infty}(\R^n\times\s)}.
\end{multline*}
  This completes the proof of the lemma.
\end{proof}
\subsubsection{Stability for the light-ray transform}
The light-ray transform $\mathcal{R}$ of a function $f\in L^1(\R^{n+1})$  is defined by
$$\mathcal{R}(f)(\theta,x):=\int_{\R}f(t,x+t\theta)dt.$$
Using the above Lemma \ref{Lemma3.1}, we can estimate the light-ray transform of the difference of the two absorption coefficients
as follows:

\begin{lemma} \label{450}
Let $  (a_j,k_j)\in \mathcal{A}^*(M_0)\times\km, \; j=1,2$, then there exists $C > 0$ such that

\begin{equation}
\left|\mathcal{R}(a)(\omega,x) \right|\leq C \Vert \A_{a_1,k_1}-\A_{a_2,k_2}  \Vert_1,\quad  \forall\;\omega \in  \s,\;\mbox{a.e},\, x\in \R^n.\label{500}
\end{equation}
\end{lemma}
Here $C$ is a constant depending only on $\Omega$, $T$ and $M_i$, $i=0,1,2$.
\begin{proof}

Let $\omega\in \s$ and $\phi_j\in\mathcal{C}^\infty_0(\mathcal{B}_r)$, $j=1,2$. Moreover, let $\varrho_h(\omega,\theta)$ defined by (\ref{4.16}). Selecting
$$
\phi^+(y,\theta)=\varrho_h(\omega,\theta)\phi_1(y),\quad \phi^-(y,\theta)=\phi_2(y)
,$$
we obtain from \eqref{ll}, that for any $h\in (0,1)$, 
\begin{multline*}
\Big|\int_{\mathcal{B}_r\times\s}\varrho_h(\omega,\theta)\phi_1(y)\phi_2(y)\Big[ \exp\Big( -\int_{0}^{T}a(s,y+s\theta)\,ds \Big)-1 \Big]\,d\theta\, dy  \Big|\cr
\leq C\|\A_{a_{2},k_{2}}-\A_{a_{1},k_{1}}\|_1\|\phi_1\|_{L^{1}(\R^n)}\|\phi_2\|_{L^{\infty}(\R^n)}\norm{\varrho_h(\omega,\cdot)}_{L^1(\s)}.
\end{multline*}
Since $\norm{\varrho_h(\omega,\cdot)}_{L^1(\s)}=1$, and by \eqref{4.19} we have
\begin{equation*}
\lim_{h\to 1}\int_{\s}\varrho_h(\omega,\theta)\Big[ \exp\Big( -\int_{0}^{T}a(s,y+s\theta)\,ds \Big)-1 \Big]\,d\theta=\Big[ \exp\Big( -\int_{0}^{T}a(s,y+s\omega)\,ds \Big)-1 \Big].
\end{equation*}
Then, we get
\begin{multline*}
\Big|\int_{\mathcal{B}_r}\phi_1(y)\phi_2(y)\Big[ \exp\Big( -\int_{0}^{T}a(s,y+s\omega)\,ds \Big)-1 \Big]\, dy  \Big|\cr
\leq C\|\A_{a_{2},k_{2}}-\A_{a_{1},k_{1}}\|_1\|\phi_1\|_{L^{1}(\Rn)}\|\phi_2\|_{L^{\infty}(\R^n)}.
\end{multline*}
Let the bilinear form $\mathscr{B}: L^1(\mathcal{B}_r)\times L^\infty(\mathcal{B}_r)\To\R$ given by
\begin{equation*}
\mathscr{B}(\phi_1,\phi_2)=\int_{\mathscr{B}_r}\phi_1(y)\phi_2(y)R(y,\omega)\, dy ,
\end{equation*}
where 
$$
R(y,\omega)=\exp\Big( -\int_{0}^{T}a(s,y+s\omega)\,ds \Big)-1
.$$
Then we have
\begin{equation*}
\abs{\mathscr{B}(\phi_1,\phi_2)}\leq C\|R(\cdot,\omega)\|_{L^\infty(\mathcal{B}_r)}\|\phi_1\|_{L^{1}(\R^n)}\|\phi_2\|_{L^{\infty}(\R^n)}.
\end{equation*}
We conclude that
\begin{equation*}
\|\mathscr{B}\|=\|R(\cdot,\omega)\|_{L^\infty(\mathcal{B}_r)}\leq C\|\A_{a_{2},k_{2}}-\A_{a_{1},k_{1}}\|_1,
\end{equation*}
where $\|\mathscr{B}\|$ stands for the norm of the bilinear form $\mathscr{B}$. That is
\begin{equation}\label{r}
 \Big|\exp\Big(-\int_{0}^{T} a(s,y+s\omega)\,ds  \Big)-1
\Big|\leq C\|\A_{a_{2},k_{2}}-\A_{a_{1},k_{1}}\|_1,\quad \forall\,y\in\mathcal{B}_r,\,\,\forall\,\omega\in\s.
\end{equation}
Using the fact that $|X|\leq e^{M}\,|e^{X}-1|$ for any $X$ real satisfying $|X|\leq M$  we  found out that 
$$\Big| \int_{0}^{T}a(s,y+s\omega)\,ds \Big|\leq e^{MT}\Big| \exp\Big( -\int_{0}^{T} a(s,y+s\omega)\,ds  \Big)-1
\Big|,$$
where  $X=\displaystyle\int_{0}^{T} a(s,y+s\omega)ds$. 
We conclude in light of (\ref{r}) the following estimate
\begin{equation}\label{eq4}
\Big|\int_{0}^{T}a(s,y+s\omega)\,ds \Big|\leq C\|\A_{a_{2},k_{2}}-\A_{a_{1},k_{1}}\|_1,\quad\forall\,y\in \mathcal{B}_r,\quad\omega\in\s.
\end{equation}
Using the fact that $a =a_2-a_1 =0$ outside $X_{r,*}$, this entails that for all $y\in \mathcal{B}_r$ and $\omega \in\s,$ we have
\begin{equation}\label{1000}
\left|\int_{\R}a(s,y+s\omega)\,ds \right|\leq C \Vert\A_{a_2,k_2}-\A_{a_1,k_1} \Vert_1.
\end{equation}
Now, if we assume $y\in B(0, r/2)$, we get $\vert y+t\omega\vert \geqslant \vert t \vert - \vert y\vert \geq \vert t\vert - r/2.$ Hence, one can see that \\ $(t,y+t\omega) \notin \mathcal{C}_r^+, \; \text{if}\; t >r/2$. On the other hand, we have $(t,y+t\omega) \notin  \mathcal{C}_r^+,$ if $ t \leqslant r/2$. Thus, we conclude that $(t,y+t\omega)\notin \mathcal{C}_r^+\supset X_{r,*}$ for $t \in \R$. This and the fact that $a =a_2-a_1=0$ outside $X_{r,*}$, entails that for all $ y\in B(0, r/2)$ and $\omega\in \s$, we have

$$a(t,y+t\omega)=0, \;\forall t\in \R .$$
By a similar way, we prove for $\vert y\vert  \geqslant T -r/2,$ that $(t,y+t\omega)\notin \mathcal{C}_r^-\supset X_{r,*}$ for $t \in \R$ and then $a(t,y+t\omega)=0$. Hence, we conclude that
\begin{equation}\label{100}
\left|\int_{\R}a(s,y+s\omega)\,ds \right|\leq C \Vert\A_{a_1,k_1}-\A_{a_2,k_2}  \Vert_1,\quad \forall\;  \omega \in  \s,\; y\notin \mathcal{B}_r.
\end{equation}
Thus, by (\ref{eq4}) and (\ref{100})we finish the proof of the lemma by getting

$$\left|\int_{\R}a(s,y+s\omega)\,ds \right|\leq C \Vert \A_{a_1,k_1}-\A_{a_2,k_2} \Vert_1,\quad \forall\;  \omega \in  \s,\; y\in \R^n.$$
 The proof of Lemma \ref{450} is complete.
\end{proof}
Our goal is to obtain an estimate linking the Fourier transform with respect to $(t, x)$ of the absorption coefficient $a =a_1-a_2$ to the measurement $(\A_{a_1,k_1}-\A_{a_2,k_2})$ in the conic set
$$E=\lbrace (\tau,\xi)\in \R\times \R^n, \; \vert\tau\vert\leq\vert\xi\vert\rbrace .$$
We denote by $\widehat{g}$ the Fourier transform of a function $g\in L^1(\R^{n+1})$ with respect to $(t, x)$:
$$\widehat{g}(\tau,\xi)=\int_{\R}\int_{\R^n}g(t,x) e^{-ix\cdot\xi}  e^{-it\tau}dxdt. $$
We aim for proving that the Fourier transform of $a$ is bounded as follows:
\begin{lemma} \label{50.}
Let $  (a_j,k_j)\in \mathcal{A}^*(M_0)\times\km, $ $\; j=1,2$, then there exists $C > 0$, such that the following estimate
\begin{equation}
\left|\widehat{ a}(\tau,\xi) \right|\leq C \Vert \A_{a_1,k_1}-\A_{a_2,k_2} \Vert_1,\label{5*00}
\end{equation}
holds for any $(\tau,\xi)\in E.$ 
\end{lemma}
\begin{proof}
Let $(\tau,\xi)\in E$ and $\zeta\in \s$ be such that $\xi\cdot\zeta=0.$ Setting
$$\omega=\frac{\tau}{\vert\xi\vert^2}\cdot \xi+\sqrt{1-\frac{\tau^2}{\vert\xi\vert^2}}\cdot\zeta .$$
Then, one can see that $\omega\in\s $ and $\omega\cdot\xi=\tau.$
On the other hand by the change of variable $x =y+t\omega$, we have for all $\xi\in \R^n$ and $\omega\in\s $, the following identity

\begin{align*}
\int_{\R^n}\mathcal{R}(a)(y,\omega)e^{-iy\cdot\xi}dy=&\int_{\R^n}\left(\int_{\R}  a(s,y+s\omega)\,ds\right) e^{-iy\cdot\xi}dy\\
&=\int_{\R^n}\int_{\R}  a(t,x) e^{-ix\cdot\xi}e^{-it\omega\cdot\xi}dx dt\\
&=\widehat{a}(\omega\cdot\xi,\xi)=\widehat{a}(\tau,\xi),
\end{align*}
where we have set $(\omega\cdot\xi,\xi)=(\tau,\xi)\in E.$  Bearing in mind that for any $t\in \R$, $\text{Supp}\; a(t,\cdot)\subset\Omega\subset B(0,r/2),$ we obtain

$$\int_{\R^n\cap B(0,\frac{r}{2}+T)}\mathcal{R}(a)(y,\omega)e^{-iy\cdot\xi}dy=\widehat{a}(\tau,\xi).$$
Then, in view of Lemma \ref{450}, we complets the proof of the Lemma \ref{50.}.
\end{proof}
 
We now give the following result which is proved in \cite{[8]}.
\begin{lemma}\label{l}
Let $\mathcal{O}$ be an open set of $B(0,1)\subset \R^{n+1}$, and $F$ an analytic function in $B(0,2),$ such that: there exist constant $M,\;\eta>0$ such that
$$\Vert\p^\gamma F  \Vert_{L^\infty(B(0,2))}\leq \frac{M \vert \gamma\vert ! }{\eta^{\vert\gamma\vert}},\quad \gamma\in (\N\cup\lbrace 0\rbrace)^{n+1}. $$  
Then,
$$ \Vert F  \Vert_{L^\infty(B(0,1))}\leq (2M)^{1-\mu}\Vert F \Vert^{\mu}_{L^\infty(\mathcal{O})}.$$
where $\mu \in (0,1)$ depends on $n,\eta$ and $\vert \mathcal{O}  \vert$. Here $\vert\gamma\vert=\gamma_1+\dots +\gamma_{n+1}\,\mbox{and}\, B(0,\rho)=\lbrace y\in \R^{n+1},\;\vert y\vert<\rho\rbrace .$
\end{lemma}
\subsubsection{End of the proof of Theorem \ref{20}}
We complete now the proof of Theorem \ref{20}.
For a fixed $\alpha>0$, we set $F_\alpha(\tau, \xi) = \widehat{ a}(\alpha(\tau, \xi))$, for all $(\tau, \xi)\in \R^{n+1}$. It is easy to see that $F_\alpha$ is analytic and we have
\begin{equation*}
\begin{split}
\vert \p^\gamma F_\alpha(\tau, \xi)\vert&=  \vert \p^\gamma\int_{\R^{n+1}} a(t,x) e^{-i \alpha x\cdot\xi}e^{-i\alpha t\omega\cdot\xi}dx dt\vert\\&
=\Big\vert\int_{\R^{n+1}} a(t,x)(-i)^{\vert \gamma\vert}\alpha^{\vert \gamma\vert}(t,x)^{ \gamma} e^{-i\alpha(t,x)\cdot(\tau,\xi)}dxdt\Big\vert.
\end{split}
\end{equation*}
This entails that

$$\vert \p^\gamma F_\alpha(\tau,\xi)\vert\leq  \int_{\R^{n+1}}\vert a(t,x) \vert  \alpha^{\vert \gamma\vert}(\vert x\vert^2+t^2)^{\frac{\vert \gamma\vert}{2}} dxdt \leq \Vert a\Vert_{L^1(X_{r,*})}\alpha^{\vert \gamma\vert}(2T^2)^{\frac{\vert \gamma\vert}{2}}\leq C\frac{\vert \gamma\vert!}{(T^{-1})^{\vert \gamma\vert}}e^{\alpha}.$$
Then, applying Lemma \ref{l} in the set $\mathcal{O}= \overset{\circ}{E}\cap B(0,1)$ with $M=C e^\alpha$ and $\eta=T^{-1},$  where

 $$\overset{\circ}{E}=\lbrace (\tau,\xi)\in \R\times \R^n, \; \vert\tau\vert<\vert\xi\vert\rbrace .$$ 
So, we can find a constant $\mu\in (0,1), $ such that we have for all $(\tau,\xi)\in B(0,1),$ the following estimation

$$\vert  F_\alpha(\tau,\xi)\vert=\vert\widehat{a}(\alpha(\tau, \xi))\vert\leq Ce^{\alpha(1-\mu)}\Vert F_\alpha\Vert^{\mu}_{L^\infty(\mathcal{O})},\quad  (\tau, \xi)\in B(0,1).$$
Now, we will find an estimate for the Fourier transform of the absorption coefficient $a$ in a suitable ball. Using the fact that $\alpha \overset{\circ}{E}=\lbrace\alpha(\tau, \xi),\;(\tau, \xi)\in \overset{\circ}{E}\rbrace=\overset{\circ}{E},$ we obtain for all $(\tau, \xi)\in B(0,\alpha)$
\begin{equation}\label{f}
\begin{split}
\vert\widehat{a}(\tau, \xi)\vert=\vert  F_\alpha(\alpha^{-1}(\tau, \xi))\vert&\leq  Ce^{\alpha(1-\mu)}\Vert F_\alpha\Vert^{\mu}_{L^\infty(\mathcal{O})}\\&\leq  Ce^{\alpha(1-\mu)}\Vert \widehat{a}\Vert^{\mu}_{L^\infty(B(0,\alpha)\cap\overset{\circ}{E})}\\&\leq C e^{\alpha(1-\mu)}\Vert \widehat{a}\Vert^{\mu}_{L^\infty(\overset{\circ}{E})}.
\end{split}
\end{equation}
Our goal now is to deduce an estimate that links the unknown coefficient $a$ to the measurement $ \A_{a_1,k_1}-\A_{a_2,k_2} $. To obtain such estimate, we need first to decompose the $H^{-1}(\R^{n+1})$ norm of $a$ into the following way
\begin{align*}
\Vert a\Vert^{2/\mu}_{H^{-1}(\R^{n+1})}&=\left(\int_{\vert (\tau, \xi)\vert<\alpha} \!(1+\vert (\tau, \xi)\vert^2)^{-1}\vert\widehat{a}(\tau, \xi) \vert^2d\xi d\tau+\int_{\vert (\tau, \xi)\vert\geq\alpha} (1+\vert (\tau, \xi)\vert^2)^{-1}\vert\widehat{a}(\tau, \xi) \vert^2d\xi d\tau\right)^{1/\mu}\\&\leq C\left( \alpha^{n+1}\Vert\widehat{a}\Vert_{L^\infty(B(0,\alpha))}+\alpha^{-2}\Vert a\Vert^2_{L^2(\R^{n+1)}}\right)^{1/\mu}. 
\end{align*}
Therefore, by using (\ref{f}) and Lemma \ref{l}, we get
\begin{align*}
\Vert a\Vert^{2/\mu}_{H^{-1}(\R^{n+1})}&\leq  C\left( \alpha^{n+1}e^{2\alpha(1-\mu)}\Vert \A_{a_1,k_1}-\A_{a_2,k_2} \Vert_1^{2\mu}+\alpha^{-2}\right)^{1/\mu}\\&\leq C \left( \alpha^{\frac{n+1}{\mu} }e^{2\frac{\alpha(1-\mu)}{\mu}}\Vert \A_{a_1,k_1}-\A_{a_2,k_2} \Vert_1^2+\alpha^{-2/\mu}\right)\\&\leq C\left( e^{N\alpha}\Vert \A_{a_1,k_1}-\A_{a_2,k_2} \Vert_1^2+\alpha^{-2/\mu}\right),
\end{align*}
where $N$ depends on  $n$ and $\mu$. Next, we minimize the right hand-side of the above inequality with respect to $\alpha$. We need to take $\alpha$ sufficiently large. So, there exists a constant $c >0$ such that if $0<\Vert \A_{a_1,k_1}-\A_{a_2,k_2} \Vert_1 <c$, and 
$$\alpha=\frac{1}{N}\vert\log\Vert \A_{a_1,k_1}-\A_{a_2,k_2} \Vert_1\vert, $$
then, we have the following estimation
\begin{equation*}
\begin{split}
\Vert a\Vert_{H^{-1}(X_{r,*})}\leq \Vert a\Vert_{H^{-1}(\R^{n+1})}&\leq  C\left( \Vert \A_{a_1,k_1}-\A_{a_2,k_2} \Vert_1+\vert\log\Vert \A_{a_1,k_1}-\A_{a_2,k_2} \Vert_1\vert^{-2/\mu}\right)^{\mu/2} \\&\leq C\left( \Vert \A_{a_1,k_1}-\A_{a_2,k_2} \Vert_1^{\mu/2}+\vert\log\Vert \A_{a_1,k_1}-\A_{a_2,k_2} \Vert_1\vert^{-1}\right). 
\end{split}
\end{equation*}
Let us now consider $\delta$ such that  $s>\frac{n-1}{2}+\delta $. By using  Sobolev's interpolation inequality, we get
\begin{align*}
\Vert a \Vert_{L^\infty(X_{r,*})}&\leq C\Vert a \Vert_{H^{\frac{n+1}{2}+\delta
}(X_{r,*})}\\
&\leq C \Vert a \Vert^{1-\beta}_{H^{-1}(X_{r,*})} \Vert a \Vert^\beta_{H^{s+1}(X_{r,*})}\\
&\leq C\Vert a\Vert_{H^{-1}(X_{r,*})}^{1-\beta},
\end{align*}
for some $\beta\in (0, 1)$. This completes the proof of Theorem \ref{20}.

\subsection{Determination of the absorption coefficients from boundary measurements and final observations}
In this section, we prove Theorem \ref{Theorem2}. We will extend the stability estimates obtained in the first case to a larger region $X_{r,\sharp}\supset X_{r,*}$. We shall consider the geometric optics solutions constructed in Section 2, associated with a function $\phi^+$ obeying $\mbox{Supp}\,\phi^+(\cdot,\theta)\cap \Omega=\emptyset$. Note that this time, we have more flexibility on the choice of thr support of the function $\phi^-$ and we don't need to assume that $\mbox{Supp}\,\phi^-(\cdot,\theta)\pm T\theta\cap \Omega=\emptyset$ anymore. We recall that the measurements  in this case are given by the following response operator

$$\begin{array}{llll}
\tilde{\A}_{a,k}:&{\mathscr{L}^-_p(\Sigma^-)}&\rightarrow &\mathscr{K}^+_p\\
\ &\ \ \ \ f&\mapsto &\tilde{\A}_{a,k}(f)= ( u_{|\Sigma^+},\, u(T,\cdot,\cdot)),
\end{array}$$ 
associated to the problem (\ref{1}) with $u_0 =0$. We denote by
$$\tilde{\A}^1_{a,k}(f)= u_{|\Sigma^+},\quad\tilde{\A}^2_{a,k}(f)= u(T,\cdot,\cdot) .$$


Here we will find that the absorption  coefficient $a$ can be stably recovered in a larger region if we further know the final data of the solution $u$ of the linear Boltzmann  equation (\ref{1}). In the rest of this section, we define $a =a_2-a_1$ in $\overline{\Omega}_T$ and $a =0$ on $\R^{n+1}\setminus \overline{\Omega}_T$. We shall first prove the following statement:

\begin{lemma}\label{Lemma5.1}
Let $(a_{j},k_j)\in\mathcal{A}^\sharp(M_0)\times\km$,  $j=1,\,2$. There exists $C>0$ such that for any  $\phi^\pm\in \C^\infty_0(\R^n,\C(\s)) $, with  Supp
$\phi^\pm(\cdot,\theta)\cap B(0, r/2)=\emptyset,$ for all  $\theta\in \s,$  the following estimate holds true
\begin{multline*}
\Big|\int_{\R^n\times\s}(\phi^+\phi^-)(y,\theta)
\Big[ \exp\Big( -\int_{0}^{T}a(s,y+s\theta)\,ds \Big)-1 \Big]\,d\theta\, dy  \Big|\cr
\leq C\|\tilde{\A}_{a_1,k_1}-\tilde{\A}_{a_2,k_2} \|_1\|\phi^+\|_{L^{1}(\R^n\times\s)}\|\phi^-\|_{L^{\infty}(\R^n\times\s)}.
\end{multline*}
 Here $C$ depends only on
$\Omega$, $T$, and $M_i,$ $i=0,1,2.$
\end{lemma}

\begin{proof}
In view of Lemma \ref{LC1}, and using the fact that Supp
$\phi^+(\cdot,\theta)\cap\Omega=\emptyset$, there exists a G.o solution
$u_\lambda^{+}$ to the equation
$$
\left\{
    \begin{array}{ll}
      \p_{t}u+\theta\cdot\nabla u+a_{2}(t,x)u=\G_{k_2}[u] (t,x,\theta) & \textrm{in}\,\,\,Q_T, \cr
      u(0,x,\theta)=0 & \textrm{in}\,\,\, Q,
    \end{array}
  \right.$$
in the following form
\begin{equation}\label{EQ32.19}
u_\lambda^{+}(t,x,\theta)=\varphi_\lambda^+(t,x,\theta)b_{a_2}(t,x,\theta)+\psi_{\lambda}^{+}(t,x,\theta),
\end{equation}
corresponding to the absorption coefficients $a_{2}$ and the scattering coefficients $k_{2}$,  where the remainder term  $\psi_\lambda^{+}(t,x,\theta)$
satisfies (\ref{4.6}), (\ref{4.66}). Next, we set
the function
$$f_{\lambda}(t,x,\theta):= u^{+}_\lambda(t,x,\theta)_{|\Sigma^-}=\varphi_\lambda^+(t,x,\theta) b_{a_2}(t,x,\theta).$$
We denote by $u_{1}$ the solution of
$$\left\{
  \begin{array}{ll}
   \p_{t}u_{1}+\theta\cdot\nabla u_{1}+a_{1}(t,x)u_{1}=\G_{k_1}[u_1] & \textrm{in}\,\,\,Q_T, \cr
    u_{1}(0,x,\theta)=0 & \textrm{in}\,\,\, Q ,\cr
    u_{1}(t,x,\theta)=f_{\lambda}(t,x,\theta) & \textrm{on}\,\,\,\Sigma^-.
  \end{array}
\right.
$$
Putting $u=u_{1}-u_\lambda^{+}$. Then, $u$ is a solution to the following system
\begin{equation}\label{2332.31}
\left\{
  \begin{array}{ll}
     \p_{t}u+\theta\cdot\nabla u+a_{1}(t,x)u=\G_{k_1}[u]+a(t,x)u^+_\lambda-\G_k[u^+_\lambda] & \textrm{in}\,\,\,Q_T, \cr
    u(0,x,\theta)=0 & \mbox{in}\,\,\,Q, \cr
    u(t,x,\theta)=0 & \mbox{on}\,\,\,\Sigma^-.
  \end{array}
\right.
\end{equation}
On the other hand, Lemma  \ref{LC2} guarantees the existence of a G.O-solution $u^{-}_\lambda$ to the adjoint problem

$$ \p_{t}u+\theta\cdot\nabla u-a_{1}(t,x)u=-\G^*_{k_1}[u](t,x,\theta) \quad \textrm{in}\,\,\,Q_T,$$
corresponding to the coefficients $a_{1}$ and $k_1$, in the form
\begin{equation}\label{EQ32.21}
u^{-}_\lambda(t,x,\theta)=\varphi_\lambda^-(t,x,\theta)b_{-a_1}(t,x,\theta)+\psi_{\lambda}^{-}(t,x,\theta),
\end{equation}
where $\psi_{\lambda}^{-}(t,x,\theta)$
satisfies (\ref{4.7}), (\ref{4.77}). We multiply the first equation of (\ref{2332.31}) by $u^{-}_\lambda$, and we integrate by parts, we obtain
\begin{multline}\label{EQ232.22}
\int_{Q_T} a(t,x) u^{+}_\lambda\,u^{-}_\lambda\, d\theta dx dt-\int_{Q_T}\G_k[u^+_\lambda](t,x,\theta)u^{-}_\lambda (t,x,\theta) d\theta\,\,dx\,dt\cr
=\int_{\Sigma^+}\theta\cdot\nu(x) (\tilde{\A}^1_{a_{2},k_{2}}-\tilde{\A}^1_{a_{1},k_{1}}) (f_{\lambda}) u_\lambda^{-}\, d\theta\,dx\,dt+\int_{\s\times \Omega} (\tilde{\A}^2_{a_{2},k_{2}}-\tilde{\A}^2_{a_{1},k_{1}}) (f_{\lambda}) u_\lambda^{-}(T,x,\theta)\, d\theta\,dx.
\end{multline}
By replacing $u^+_\lambda$ and $u^-_\lambda$ by their expressions, we obtain
\begin{multline*}
\Big|\int_{\R^{n}\times\s}(\phi^+\phi^-)(y,\theta)\Big[ \exp\Big( -\int_{0}^{T}a(s,y+s\theta)ds \Big)-1 \Big]d\theta dy  \Big|\cr
\leq C\left[ \|(\tilde{\A}^1_{a_{2},k_{2}}-\tilde{\A}^1_{a_{1},k_{1}})f_\lambda\|_{\mathscr{L}^+_p(\Sigma^+)}+\|(\tilde{\A}^2_{a_{2},k_{2}}-\tilde{\A}^2_{a_{1},k_{1}})f_\lambda\|_{L^p(Q)}\|\right] \left[ \|u_\lambda^-\|_{\mathscr{L}^+_q(\Sigma^+)}+\|u_\lambda^-(T)\|_{L^q(Q)}\right] \\
\leq   C
\|\tilde{\A}_{a_{2},k_{2}}-\tilde{\A}_{a_{1},k_{1}}\|_p\,\|f_\lambda\|_{\mathscr{L}^-_p(\Sigma^-)}\|z^-_\lambda\|_{\mathscr{K}^+_q},
\end{multline*}
where 
$$z^-_{\lambda}=(u^-_{\lambda|\Sigma^+},u^-_{\lambda}(T,\cdot,\cdot) ).$$ 
Then, by using (\ref{inq1}) and (\ref{inq2}) and the fact that
\begin{equation}
\begin{split}
\Vert u_\lambda^-(T,\cdot,\cdot)\Vert_{L^q(Q)}&\leq C \left(\Vert\varphi_{\lambda}^{-}(T,\cdot,\cdot)\Vert_{L^q(Q)}+\Vert\psi_{\lambda}^{-}(T,\cdot,\cdot)\Vert_{L^q(Q)} \right)\\&\leq C\Vert \phi^\pm\Vert_{L^\infty(\R^n\times\s)}.
\end{split}
\end{equation}
We complete the proof of Lemma \ref{Lemma5.1}.
\end{proof} 
\begin{lemma} \label{50}
Let $  (a_j,k_j)\in \mathcal{A}^\sharp(M_0)\times\km, \; j=1,2$, then there exists $C > 0$ such that
 
\begin{equation}
\left|\mathcal{R}(a)(\theta,x) \right|\leq C \Vert \tilde{\A}_{a_1,k_1}-\tilde{\A}_{a_2,k_2}  \Vert_1,\quad  \forall\;\theta \in  \s,\;\, x\in \R^n.\label{5-00}
\end{equation}
\end{lemma}
\begin{proof}
We consider $\phi^\pm$, such that $\textrm{Supp}\,\phi^\pm(\cdot,\theta)\cap B\left( 0,\frac{r}{2}\right)$. Using Lemma \ref{Lemma5.1} and repeating the arguments used in Lemma \ref{450}, we obtain this estimate 

\begin{equation}\label{1.000}
\left|\int_{\R}a(s,y+s\theta)ds \right|\leq C \Vert\tilde{\A}_{a_1,k_1}-\tilde{\A}_{a_2,k_2}  \Vert_1,\quad \forall\;  \theta \in  \s,\; y\notin B\left( 0,\frac{r}{2}\right).
\end{equation}
Let now $y\in B\left( 0,\frac{r}{2}\right) $, let us show that $ 
a(t,y+t\theta)=0,\,\mbox{for any}\,t\in \R .$
Indeed, we have
\begin{equation*}\label{4587}
\vert y+t\theta \vert\geq \vert t\vert -\vert y\vert .
\end{equation*}
So, we deduce that for all $t>\frac{r}{2}$ we have $(t,y+t\theta)\notin \mathcal{C}_r^+$. And if $t\leq \frac{r}{2},$ we have also that $(t,y+t\theta)\notin \mathcal{C}_r^+.$ Since $X_{r,\sharp}=X\cap \mathcal{C}_r^+$. We have
$(t,y+t\theta)\notin X_{r,\sharp},$ for any  $t\in \R.$
Then using the fact that $a=a_2-a_1=0$ outside $X_{r,\sharp}$ we obtain,
\begin{equation}\label{lll}
\int_\R a(t,y+t\theta)dt=0,\quad \forall \;y\in B\left( 0,\frac{r}{2}\right).
\end{equation}
In light of (\ref{1.000}) and (\ref{lll}), the proof of Lemma \ref{50} is completed.
Using the above result and in the same way as in Section 4, we complete the proof of Theorem \ref{Theorem2}.
\end{proof}
\subsection{Determination of the absorption coefficient from boundary measurements and final observation by varying the initial data}
In this section we deal with the same problem treated in Section 4 and 5, except our data will be the
response of the medium for all possible initial data. As usual, we will prove Theorem \ref{Theorem3} using geometric
optics solutions constructed in Section 3 and X-ray transform. Let's first recall the definition of the
operator $\mathcal{I}_{a,k}$

$$\begin{array}{llll}
\mathcal{I}_{a,k}:  & \ \ \ \ \quad \mathscr{K}^-_p&\longrightarrow &\mathscr{K}^+_p\\
 & g=(f,u_0)&\longmapsto &\mathcal{I}_{a,k}(g)= (u_{|\Sigma^+},u(T,\cdot,\cdot)).
\end{array}$$ 

We denote by
$$\mathcal{I}_{a,k}^1(g)= u_{|\Sigma^+},\quad\mathcal{I}_{a,k}^2(g)= u(T,\cdot,\cdot). $$

\begin{lemma} \label{503}
Let $  (a_j,k_j)\in \F,\; j=1,2$, then there exists $C > 0$ such that

\begin{equation}
\left|\mathcal{R}(a)(\theta,x) \right|\leq C \Vert\mathcal{I}_{a_1,k_1}-\mathcal{I}_{a_2,k_2}  \Vert_1,\quad  \quad  \forall\;\theta \in  \s,\, x\in \R^n.\label{5003}
\end{equation}
\end{lemma}

\begin{proof}
Let $\phi^\pm\in \C^\infty_0(\R^n,\C(\s)) .$ For $\lambda$ sufficiently large, Lemma \ref{LC2} guarantees the existence of G.O-solution $u_\lambda^+$ to

$$
\p_{t}u+\theta\cdot\nabla u+a_{2}(t,x)u=\G_{k_2}[u] (t,x,\theta) $$
in the following form
\begin{equation}\label{EQ3.19}
u_\lambda^{+}(t,x,\theta)=\varphi_\lambda^+(t,x,\theta)b_{a_2}(t,x,\theta)+\psi_{\lambda}^{+}(t,x,\theta),
\end{equation}
corresponding to the coefficients $a_{2}$ and $k_{2}$,  here the remainder term $\psi_\lambda^{+}(t,x,\theta)$
satisfies (\ref{4.6}), (\ref{4.66}). Now, we denote by $f_{\lambda}$
the function
$$f_{\lambda}(t,x,\theta):= u^{+}_\lambda(t,x,\theta)_{|\Sigma^-}=\varphi_\lambda^+(t,x,\theta) b_{a_2}(t,x,\theta)\quad\mbox{and}\quad u_0(x,\theta)= u_\lambda^+(0,x,\theta).$$
We remark here that $u_0$ not necessary zero since \eqref{3.2} not necessary satisfied.\\
We denote by $u_{1}$ the solution of
$$\left\{
  \begin{array}{ll}
   \p_{t}u_{1}+\theta\cdot\nabla u_{1}+a_{1}(t,x)u_{1}=\G_{k_1}[u_1] & \textrm{in}\,\,\,Q_T, \cr
    u_{1}(0,x,\theta)=u_0(x,\theta) & \textrm{in}\,\,\, Q, \cr
    u_{1}(t,x,\theta)=f_{\lambda}(t,x,\theta) & \textrm{on}\,\,\,\Sigma^-.
  \end{array}
\right.
$$
Putting $u=u_{1}-u_\lambda^{+}$. Then, $u$ is a solution to the following system
\begin{equation}\label{33.31}
\left\{
  \begin{array}{ll}
     \p_{t}u+\theta\cdot\nabla u+a_{1}(t,x)u=\G_{k_1}[u]+a(t,x)u^+_\lambda-\G_k[u^+_\lambda] & \textrm{in}\,\,\,Q_T, \cr
    u(0,x,\theta)=0 & \mbox{in}\,\,\,Q, \cr
    u(t,x,\theta)=0 & \mbox{on}\,\,\,\Sigma^-,
  \end{array}
\right.
\end{equation}
 where $a=a_{2}-a_{1}$ and $k=k_{2}-k_{1}$. Applying Lemma \ref{LC2}, once more for $\lambda$ large enough, we may find a G.O-solution $u^{-}_\lambda$ to the backward problem of (\ref{1})
$$
\p_{t}u+\theta\cdot\nabla u-a_{1}(t,x)u=-\G^*_{k_1}[u](t,x,\theta),
$$
corresponding to the coefficients $a_{1}$ and $k_1$, in the form
\begin{equation}\label{EQ3*.21}
u^{-}_\lambda(t,x,\theta)=\varphi_\lambda^-(t,x,\theta)b_{-a_1}(t,x,\theta)+\psi_{\lambda}^{-}(t,x,\theta),
\end{equation}
 where $\psi_{\lambda}^{-}(t,x,\theta)$
satisfies (\ref{4.7}), (\ref{4.77}). Multiplying the equation  (\ref{33.31}) by $u^{-}_\lambda$, integrating by parts and using Green's formula, we get
\begin{multline}\label{EQ3.22}
\int_{Q_T} a(t,x) u^{+}_\lambda\,u^{-}_\lambda\, d\theta dx dt-\int_{Q_T}\G_k[u^+_\lambda](t,x,\theta)u^{-}_\lambda (t,x,\theta) d\theta\,\,dx\,dt\cr
=\int_{\Sigma^+}\theta\cdot\nu(x) (\mathcal{I}^1_{a_{2},k_{2}}-\mathcal{I}^1_{a_{1},k_{1}}) (g_{\lambda}) u_\lambda^{-}\, d\theta\,dx\,dt+\int_{\s\times \R^n} (\mathcal{I}^2_{a_{2},k_{2}}-\mathcal{I}^2_{a_{1},k_{1}}) (g_{\lambda}) u_\lambda^{-}(T,x,\theta)\, d\theta\,dx,
\end{multline}
where
$$g_{\lambda}=(u^+_{\lambda|\Sigma^-},u^+_{\lambda}(0,\cdot,\cdot) ).$$
Now, replacing $u_\lambda^{+}$ and $u_\lambda^{-}$ by their expressions and  proceeding as in the proof of Lemma \ref{Lemma5.1}, we get
\begin{multline*}
\Big|\int_{\R^n\times\s}(\phi^+\phi^-)(y,\theta)
\Big[ \exp\Big( -\int_{0}^{T}a(s,y+s\theta)\,ds \Big)-1 \Big]\,d\theta\, dy  \Big|\cr
\leq C\|\mathcal{I}_{a_{2},k_{2}}-\mathcal{I}_{a_{1},k_{1}}\|_1\|\phi^+\|_{L^{1}(\R^n\times\s)}\|\phi^-\|_{L^{\infty}(\R^n\times\s)}.
\end{multline*}
Now, in order to complete the proof of Lemma \ref{503}, it will be enough to fix $y \in \R^n$, consider $\phi^+$
defined as before, and proceed as in the proof of Lemma \ref{50}. By repeating the arguments used in the
previous sections, we complete the proof of Theorem \ref{Theorem3}.
\end{proof}
\section{Identification of the scattering coefficient} 
\setcounter{equation}{0}
Let $\omega\in \s$ and $\phi_j\in\mathcal{C}^\infty_0(\mathcal{B}_r)$, $j=1,2$. Moreover, let $\varrho_h(\omega,\cdot)$ defined by (\ref{4.16}). Selecting
$$
\phi^+(y,\theta)=\varrho_h(\omega,\theta)\phi_1(y),\quad \phi^-(y,\theta)=\varrho_{h'}(\omega,\theta)\phi_2(y)
.$$
Moreover, we denote
\begin{equation*}
\varphi_{h,\lambda}^+(t,x,\theta)=\varphi_{1,h}(t,x,\theta)e^{i\lambda(t-x\cdot\theta)},\quad \varphi_{h',\lambda}^-(t,x,\theta)=\varphi_{2,h'}(t,x,\theta)e^{-i\lambda(t-x\cdot\theta)},
\end{equation*}
where
$$
\varphi_{1,h}(t,x,\theta)=\varrho_h(\omega,\theta)\phi_1(x-t\theta),\quad \varphi_{2,h'}(t,x,\theta)=\varrho_{h'}(\omega,\theta)\phi_2(x-t\theta).
$$
Finally, we set
$$
b_a(t,x,\theta)=\exp\para{-\int_0^t a(s,x-s\theta)ds},\quad (t,x,\theta)\in\R\times\Rn\times\s,
$$
where we have extended $a$ by $0$ outside $\Omega$.\\
By Lemma \ref{LC1}, the following system
\begin{equation}
\begin{array}{lll}
\partial_tu+\theta\cdot\nabla u+a_2(t,x)u=\G_{k_2}[u](t,x,\theta)  & \,\,\textrm{in}\,Q_T, \cr
u(0,x,\theta)=0 & \,\, \textrm{in}\, Q,
\end{array}
\end{equation}
has a solution of the form
$$
u^+_{h,\lambda}(t,x,\theta)=\varphi_{h,\lambda}^+(t,x,\theta)b_{a_2}(t,x,\theta)+\psi^+_{h,\lambda}(t,x,\theta),
$$
satisfying
$$
u\in \mathcal{C}([0,T];W_p(Q))\cap \mathcal{C}^1([0,T];L^p(Q)),
$$
where  $\psi^+_{h,\lambda}$ satisfies
\begin{equation}\label{3*.26}
\left\{
\begin{array}{llll}
\partial_t\psi^+_{h,\lambda}+\theta\cdot\nabla\psi^+_{h,\lambda}+a_2(t,x)\psi^+_{h,\lambda}=\G_{k_2}[\psi^+_{h,\lambda}]+\G_{k_2}[\varphi_{h,\lambda}^+b_{a_2}](t,x,\theta)
& \textrm{in }\,\, Q_T,\cr
\psi^+_{h,\lambda}(0,x,\theta)=0,& \textrm{in}\,\,Q,\cr
\psi^+_{h,\lambda}(t,x,\theta)=0 & \textrm{on} \,\, \Sigma^-.
\end{array}
\right.
\end{equation}
We denote by $u_{1,h}$ the solution of
$$\left\{
  \begin{array}{ll}
   \p_{t}u_{1,h}+\theta\cdot\nabla u_{1,h}+a_{1}(t,x)u_{1,h}=\G_{k_1}[u_{1,h}] & \textrm{in}\,\,\,Q_T, \cr
    u_{1,h}(0,x,\theta)=0 & \textrm{in}\,\,\, Q \cr
    u_{1,h}(t,x,\theta)=f_{h,\lambda} (t,x,\theta):=u^+_{h,\lambda} (t,x,\theta)& \textrm{on}\,\,\,\Sigma^-.
  \end{array}
\right.
$$
Putting $u=u_{1,h}-u_{h,\lambda}^{+}$. Then, $u$ is a solution to the following system
\begin{equation}\label{EQ3.20}
\left\{
  \begin{array}{ll}
     \p_{t}u+\theta\cdot\nabla u+a_{1}(t,x)u=\G_{k_1}[u]+a(t,x)u^+_{h,\lambda}-\G_k[u^+_{h,\lambda} ] & \textrm{in}\,\,\,Q_T, \cr
    u(0,x,\theta)=0 & \mbox{in}\,\,\,Q, \cr
    u(t,x,\theta)=0 & \mbox{on}\,\,\,\Sigma^-,
  \end{array}
\right.
\end{equation}
 where $a=a_{2}-a_{1}$ and $k=k_{2}-k_{1}$. On the other hand by Lemma \ref{LC2}, the following system
\begin{equation*}
\begin{array}{lll}
 \p_{t}u+\theta\cdot\nabla u-a_{1}(t,x)u=-\G^*_{k_1}[u](t,x,\theta) & \textrm{in}\,\,\,Q_T, \cr
    u(T,x,\theta)=0 & \textrm{in}\,\,\,\, Q,
  \end{array}
\end{equation*}
has a solution of the form
$$
u^{-}_{h',\lambda}(t,x,\theta)=\varphi_{h',\lambda}^-(t,x,\theta)b_{-a_1}(t,x,\theta)+\psi_{h',\lambda}^{-}(t,x,\theta),$$
satisfying
$$u\in \mathcal{C}([0,T];W_q(Q)) \cap \mathcal{C}^1([0,T];L^q(Q)),$$
 where  $\psi^-_{h',\lambda}$ satisfies
\begin{equation}\label{3--.26}
\left\{
\begin{array}{llll}
\partial_t\psi^-_{h',\lambda}+\theta\cdot\nabla\psi^-_{h',\lambda}-a_1(t,x)\psi^-_{h',\lambda}=-\G_{k_1}^*[\psi^-_{h',\lambda}]-\G_{k_1}^*[\varphi^-_{h',\lambda}b_{a_1}](t,x,\theta)
& \textrm{in }\,\, Q_T,\cr
\psi^-_{h',\lambda}(T,x,\theta)=0,& \textrm{in}\,\,Q,\cr
\psi^-_{h',\lambda}(t,x,\theta)=0 & \textrm{on} \,\, \Sigma^+.
\end{array}
\right.
\end{equation}
  Multiplying the first equation of (\ref{EQ3.20}) by $u^{-}_\lambda$, using Green's formula and integrating by parts, we find
\begin{multline}\label{EQ3*.22}
\int_{Q_T} a(t,x) u^{+}_{h,\lambda}\,u^{-}_{h',\lambda}\, d\theta dx dt-\int_{Q_T}\G_k[u^+_{h,\lambda}](t,x,\theta)u^{-}_{h',\lambda} (t,x,\theta) d\theta\,\,dx\,dt\cr
=\int_{\Sigma^+}\theta\cdot\nu(x) (\A_{a_{2},k_{2}}-\A_{a_{1},k_{1}}) (f_{h,\lambda}) u_{h',\lambda}^{-}\, d\theta\,dx\,dt.
\end{multline}
Since we are assuming that $\A_{a_1,k_1}=\A_{a_2,k_2}$, then from Corollary \ref{c} we have $a_1=a_2$ in $X_{r,*}$ and we get $a_1=a_2$ in $\overline{\Omega}_T$. As a consequence, the identity  \eqref{EQ3*.22} is reduced to

$$\int_{Q_T}\G_k[u^+_{h,\lambda}](t,x,\theta)u^{-}_{h',\lambda} (t,x,\theta) d\theta\,\,dx\,dt
=0.$$
By changing $u^+_\lambda$ and $u^{-}_\lambda$ by their expressions, we find 
\begin{multline}\label{5.78}
-\int_{Q_T} \G_k[\varphi_{h,\lambda}^+b_{a_2}](t,x,\theta) \varphi_{h',\lambda}^-(t,x,\theta)b_{-a_1}(t,x,\theta)d\theta\,dx\,dt\cr
=\int_{Q_T} \G_k[\psi_{h,\lambda}^{+}](t,x,\theta)\varphi_{h',\lambda}^-(t,x,\theta) b_{-a_1}(t,x,\theta)d\theta\,dx\,dt\cr
+\int_{Q_T} \G_k[\varphi_{h,\lambda}^+b_{a_2}](t,x,\theta)\psi_{h',\lambda}^{-}(t,x,\theta)d\theta\,dx\,dt\cr
+\int_{Q_T} \G_k[\psi_{h,\lambda}^{+}](t,x,\theta)\psi_{h',\lambda}^{-}(t,x,\theta)d\theta\,dx\,dt.
\end{multline}
We denote the left and the right terms of  \eqref{5.78} by 
\begin{equation}\label{G}
\mathcal{G}_{\lambda,h,h'}(\omega)=
\int_{Q_T} \G_k[\varphi_{h,\lambda}^+b_{a_2}](t,x,\theta) \varphi_{
h',\lambda}^-(t,x,\theta)b_{-a_1}(t,x,\theta)d\theta\,dx\,dt,
\end{equation}
and
\begin{multline}\label{II}
I_{\lambda,h,h'}(\omega)=\int_{Q_T} \G_k[\psi_{h,\lambda}^{+}](t,x,\theta)\varphi_{h',\lambda}^-(t,x,\theta) b_{-a_1}(t,x,\theta)d\theta\,dx\,dt\cr
+\int_{Q_T} \G_k[\varphi_{h,\lambda}^+b_{a_2}](t,x,\theta)\psi_{h',\lambda}^{-}(t,x,\theta)d\theta\,dx\,dt\
+\int_{Q_T} \G_k[\psi_{h,\lambda}^{+}](t,x,\theta)\psi_{h',\lambda}^{-}(t,x,\theta)d\theta\,dx\,dt .
\end{multline}
\begin{lemma}\label{l5.2}
Let $\omega\in\s$ and $\psi^+_{h,\lambda}$ the unique solution to \eqref{3*.26}. We denote by  $\zeta^+_{\lambda,\omega}$ be the solution of the following problem
\begin{equation}
\left\{
\begin{array}{llll}
\partial_t\zeta^+_{\lambda,\omega}+\theta\cdot\nabla\zeta^+_{\lambda,\omega}+a_2(t,x)\zeta^+_{\lambda,\omega}=\G_{k_2}[\zeta^+_{\lambda,\omega}]+k_2(x,\theta,\omega)\varphi_\lambda^+(t,x,\omega)b_{a_2}(t,x,\omega) & \textrm{in }\,\, Q_T,\cr
\zeta^+_{\lambda,\omega}(0,x,\theta)=0& \textrm{in}\,\,Q,\cr
\zeta^+_{\lambda,\omega}(t,x,\theta)=0 & \textrm{on} \,\, \Sigma^-,
\end{array}
\right.
\end{equation}
where $\varphi_\lambda^+$ is given by \eqref{2.5}. Then we have 
\begin{equation} 
\lim_{h\to 1}\norm{\psi^+_{h,\lambda}-\zeta^+_{\lambda,\omega}}_{L^2(Q_T)}=0,
\end{equation}
Moreover, we get
\begin{equation} 
\lim_{h\to 1}\norm{\G_{k_2}[\psi^+_{h,\lambda}]-\G_{k_2}[\zeta^+_{\lambda,\omega}]}_{L^2(Q_T)}=0.
\end{equation}
\end{lemma}
\begin{proof}
We denote $w_{h,\lambda}(t,x,\theta)=\psi_{h,\lambda}^+-\zeta^+_{\lambda,\omega}$, then we get
\begin{equation}\label{3.26}
\left\{
\begin{array}{ll}
\partial_t w_{h,\lambda}+\theta\cdot\nabla w_{h,\lambda}+a_2(t,x) w_{h,\lambda}=\G_{k_2}[w_{h,\lambda}]+\G_{k_2}[\varphi_{h,\lambda}^+b_{a_2}] - S_{\lambda,\omega}(t,x,\theta)
& \textrm{in }\,\, Q_T,\cr
w_{h,\lambda}(0,x,\theta)=0& \textrm{in}\,\,Q,\cr
w_{h,\lambda}(t,x,\theta)=0 & \textrm{on} \,\, \Sigma^-,
\end{array}
\right.
\end{equation}
where 
$$
S_{\lambda,\omega}(t,x,\theta)=k_2(x,\theta,\omega)\varphi_\lambda^+(t,x,\omega)b_{a_2}(t,x,\omega).
$$
Then we get by Lemma \ref{L.1.1}
\begin{equation*}
\norm{w_{h,\lambda}}_{L^2(Q_T)}\leq C\norm{\G_{k_2}[\varphi_{h,\lambda}^+b_{a_2}]-S_{\lambda,\omega}}_{L^2(Q_T)},
\end{equation*}
where $C$ is a constant depending only on $\Omega$, $T$.
Moreover, we have by \eqref{4.19}
$$
\lim_{h\to 1}\G_{k_2}[(\varphi_{h,\lambda}^+)b_{a_2}](t,x,\theta)=S_{\lambda,\omega}(t,x,\theta),\quad \text{in}\,L^2(\s).$$
We deduce that

\begin{equation*}
\lim_{h\to 1}\norm{\G_{k_2}[(\varphi_{h,\lambda}^+b_{a_2}]-S_{\lambda,\omega}}_{L^2(Q_T)}=0,
\end{equation*}
which implies that
\begin{equation*} 
\lim_{h\to 1}\norm{\psi^+_{h,\lambda}-\zeta^+_{\lambda,\omega}}_{L^2(Q_T)}=0.
\end{equation*}
This completes the proof of Lemma.\\
We prove by proceeding as the proof of Lemma \ref{l5.2} the flowing Lemma.

\end{proof}
\begin{lemma}\label{l55.2}
Let $\omega\in \s$ and $\psi^-_{h',\lambda}$ the unique solution to \eqref{3--.26}. Let $\zeta^-_{\lambda,\omega}$  be the solution of the following problem
\begin{equation}\label{33.26}
\left\{
\begin{array}{llll}
\partial_t\zeta^-+
\theta\cdot\nabla\zeta^- -a_1(t,x)\zeta^-=-\G_{k_1}^*[\zeta^-]-k_1(x,\theta,\omega)\varphi_\lambda^-(t,x,\omega)b_{a_1}(t,x,\omega) & \textrm{in }\,\, Q_T,\cr
\zeta^-(T,x,\theta)=0& \textrm{in}\,\,Q,\cr
\zeta^-(t,x,\theta)=0 & \textrm{on} \,\, \Sigma^+,
\end{array}
\right.
\end{equation}
where $\varphi_\lambda^-$ is given by (\ref{2.5}). Then we have
\begin{equation} 
\lim_{h'\to 1}\norm{\psi^-_{h',\lambda}-\zeta^-_{\lambda,\omega}}_{L^2(Q_T)}=0.
\end{equation}
Moreover, we get  
\begin{equation} 
\lim_{h'\to 1}\norm{\G_{k_1}^*[\psi^-_{h',\lambda}]-\G_{k_1}^*[\zeta^-_{\lambda,\omega}]}_{L^2(Q_T)}=0.
\end{equation}
\end{lemma}
\begin{lemma}\label{5.3}
Let $\mathcal{G}_{\lambda,h,h'}$ and $I_{\lambda,h,h'},$ are given respectively by \eqref{G} and \eqref{II}. Then, we have 
$$\displaystyle \lim_{\underset{h'\to 1}{h \to 1}} \mathcal{G}_{\lambda,h,h'}(\omega)= \mathcal{G} (\omega), \quad \forall\omega\in \s, $$
and
$$\displaystyle \lim_{\underset{h'\to 1}{h \to 1}} I_{\lambda,h,h'}(\omega)= I_{\lambda} (\omega), \quad \forall\omega\in \s, $$
where $\mathcal{G}$ and $I_\lambda$ are respectively given by 
\begin{equation}
\begin{split}
\mathcal{G}(\omega)=\int_0^T\int_\Omega k(x,\omega,\omega)\phi_1(x-t\omega)\phi_2(x-t\omega)dxdt,
\end{split}\label{4.31}
\end{equation}

\begin{multline}\label{IL} 
I_\lambda(\omega)=\int_0^T\int_\Omega \left( \G_{k}[\zeta^+_{\lambda,\omega}]\phi_2(x-t\omega)b_{-a_1}e^{-i\lambda(t-\omega\cdot x)}+\G^*_{k}[\zeta^-_{\lambda,\omega}]\phi_1(x-t\omega)b_{a_2}e^{i\lambda(t-\omega\cdot x)}\right) dx\,dt\\+\int_{Q_T}\G_{k}[\zeta^+_{\lambda,\omega}]\zeta^-_{\lambda,\omega}d\theta \,dx\,dt.
\end{multline}

\end{lemma}

\begin{proof}
By replacing $\varphi_{h,\lambda}^+$ and $\varphi_{h',\lambda}^-$  and by using (\ref{4.17}) and  \eqref{4.18}, we obtain

$$
\vert \G_{k}[\varphi_{h,\lambda}^+b_{a_2}](t,x,\theta) \varphi_{h',\lambda}^-(t,x,\theta)b_{-a_1}(t,x,\theta)\vert\leq C (1-h')^{1-n},$$
where $C$ depend on $n,$ $T$ and  $M_i$, $ i=0,1,2.$  
Therefore, by using (\ref{4.19}) and Lebesgue's Theorem, one gets
\begin{equation}
\lim_{h\to 1}\int_{Q_T}\G_{k}[\varphi_{h,\lambda}^+b_{a_2}](t,x,\theta) \varrho_{h'}(\omega,\theta)\phi_2(x-t\theta)e^{ -i\lambda(t-x\cdot\theta)} b_{-a_1}(t,x,\theta)d\theta\,dx\,dt\\=\mathcal{G}_{\lambda,h'}(\omega).
\end{equation}
where
$$\mathcal{G}_{\lambda,h'}(\omega)=\int_{Q_T}k(x,\theta,\omega)\phi_1(x-t\omega)e^{ -i\lambda x\cdot\omega} b_{a_2}(t,x,\omega) \varrho_{h'}(\omega,\theta)\phi_2(x-t\theta)e^{ i\lambda x\cdot\theta}b_{-a_1}(t,x,\theta)d\theta\,dx\,dt .$$
By a similar way, we obtain that
$$\lim_{h'\to 1}\mathcal{G}_{\lambda,h'}(\omega)=\mathcal{G}(\omega),\quad \;\forall\omega\in\s,$$
where we have used $a_2=a_1.$\\We may write $I_{\lambda,h,h'}(\omega)$ as $I_{\lambda,h,h'}(\omega)=I^1_{\lambda,h,h'}(\omega)+I^2_{\lambda,h,h'}(\omega)+I^3_{\lambda,h,h'}(\omega),$ where 
\begin{align*}
I^1_{\lambda,h,h'}(\omega)=&\int_{Q_T} \G_{k}[\psi_{h,\lambda}^{+}](t,x,\theta)\varphi_{h',\lambda}^-(t,x,\theta) b_{-a_1}(t,x,\theta)d\theta\,dx\,dt\cr
=&\int_{Q_T} \G_{k'}[\psi_{h,\lambda}^{+}](t,x,\theta)\varrho_{h'}(\omega,\theta)\phi_2(x-t\theta)e^{ -i\lambda(t-x\cdot\theta)} b_{-a_1}(t,x,\theta)d\theta\,dx\,dt,\\
I^2_{\lambda,h,h'}(\omega)=&\int_{Q_T} \G_{k}[\varphi_{h,\lambda}^+b_{a_2}](t,x,\theta)\psi_{h',\lambda}^{-}(t,x,\theta)d\theta\,dx\,dt\\
=&\int_{Q_T} \G_{k}[\varrho_h(\omega,\theta)\phi_1(x-t\theta)e^{ i\lambda(t-x\cdot\theta)}b_{a_2}](t,x,\theta)\psi_{h',\lambda}^{-}(t,x,\theta)d\theta\,dx\,dt,\\
I^3_{\lambda,h,h'}(\omega)=&\int_{Q_T} \G_{k}[\psi_{h,\lambda}^{+}](t,x,\theta)\psi_{h',\lambda}^{-}(t,x,\theta)d\theta\,dx\,dt.
\end{align*}
Taking the limit as $h\to 1$ in the above expressions, we get from Lemma \ref{l5.2}  and \eqref{4.19} 
\begin{align*}
\lim_{h\to 1} I^1_{\lambda,h,h'}(\omega)=&\int_{Q_T} \G_{k}[\zeta^+_{\lambda,\omega}](t,x,\theta)\varrho_{h'}(\omega,\theta)\phi_2(x-t\theta)e^{ -i\lambda(t-x\cdot\theta)}b_{-a_1}(t,x,\theta)d\theta\,dx\,dt,\\
\lim_{h\to 1} I^2_{\lambda,h,h'}(\omega)=&\int_{Q_T} k(x,\theta,\omega)\phi_1(x-t\omega)e^{ i\lambda(t-x\cdot\omega)} b_{a_2}(t,x,\omega)\psi_{h',\lambda}^{-}(t,x,\theta)d\theta\,dx\,dt, \\
\lim_{h\to 1} I^3_{\lambda,h,h'}(\omega)=&\int_{Q_T} \G_{k}[\zeta_{\lambda,\omega}^{+}](t,x,\theta)\psi_{h',\lambda}^{-}(t,x,\theta)d\theta\,dx\,dt.
\end{align*}
We repeat the same procedure with $h'\to 1$ to obtain  $\displaystyle\lim_{\underset{h'\to 1}{h \to 1}} I_{\lambda,h,h'}(\omega)= I_{\lambda} (\omega), \,\text{for any}\,\omega\in \s.$

This completes the proof of the lemma.
\end{proof}
\begin{lemma}\label{5.4}
Let  $I_\lambda$ given by \eqref{IL}. Then, we have  
$$\lim_{\lambda\to \infty}I_\lambda(\omega)=0,\quad\forall\omega\in\s.$$
\end{lemma}
\begin{proof}
By using the fact that $\zeta_{\lambda,\omega}^{\pm}\rightharpoonup 0$ in $L^2(\Omega)$ and that the integral operator $\G_{k}$ is compact, we get
\begin{equation*}
\lim_{\lambda\to\infty}\norm{\G_{k}[\zeta_{\lambda,\omega}^{+}](t,\cdot,\theta)}_{L^2(\Omega)}=\lim_{\lambda\to\infty}\norm{\G_{k}^*[\zeta_{\lambda,\omega}^{-}](t,\cdot,\theta)}_{L^2(\Omega)}=0, \quad \forall t\in (0,T),\; \theta\in \s .
\end{equation*}
The Lebesgue Theorem implies that
\begin{equation*}
\lim_{\lambda\to\infty}\norm{\G_{k}[\zeta_{\lambda,\omega}^{+}]}_{L^2(Q_T)}=\lim_{\lambda\to\infty}\norm{\G_{k}^*[\zeta_{\lambda,\omega}^{-}]}_{L^2(Q_T)}=0.
\end{equation*}
This complete the proof.
\end{proof}
We are now in position to prove our main result.

Since we have by (\ref{5.78}), that $\mathcal{G}_{\lambda,h,h'} (\omega ) = I_{\lambda,h,h'} (\omega )$ for all $\omega\in \s$, we take the limit in $h , h'$ and $\lambda$ and we get, from Lemmas \ref{5.3} and \ref{5.4},  that

$$ \int_0^T\int_{\R^n} k(x,\omega,\omega) \phi_1(x-t\omega)\phi_2(x-t\omega)dxdt=0.$$
where $k$ is extended by $0$ outside   $\Omega.$  Since $k_j(x,\omega,\omega)=\rho_j(x)\kappa(\omega,\omega)$, $j=1,2,$ where $\kappa\in L^\infty(\s \times\s)$ and $\kappa(\omega,\omega)\neq 0$ a.e on $\s$.
Using the change of variables, $y=x-t\omega,$  for $\rho=\rho_2-\rho_1$, we get 
$$\int_{0}^T\int_{\R^n}   
\rho(y+t\omega)\kappa(\omega,\omega)
\phi_1(y)\phi_2(y)dydt=0,\,\,\forall\,
\phi_1,\phi_2\in \C_0^\infty(\mathcal{B}_r),\,\,\forall\,\omega\in \s.$$
Put
$\phi_1(x)=\phi_2(x)$, we obtain
\begin{equation}\label{0258}
\int_{0}^T\int_{\R^n}
\rho(y+t\omega)\kappa(\omega,\omega)
\phi_1^2(y)dydt=0,\,\,\forall\,
\phi_1\in \C_0^\infty(\mathcal{B}_r),\,\,\forall\,\omega\in \s.
\end{equation}
Now, we consider a non negative function   $\phi_0\in \C_0^\infty(B(0,1))$
such that $ \Vert \phi_0 \Vert_{L^2(\R^n)} = 1$. We define
$$\phi_h(x)=h^{-n/2}\phi_0\left( \frac{x-y}{h}\right) , $$ 
where $y\in \mathcal{B}_r$. Then, for $h >0$ sufficiently small we obtain
$$
\textrm{Supp} \, \phi_h\cap\Omega=\emptyset.
\quad (\textrm{Supp}\, \phi_h \pm T\theta)\cap\Omega=\emptyset.
$$
We rewrite (\ref{0258}) with $\phi_1=\phi_h$, we deduce that


$$\lim_{h\mapsto 0}\displaystyle\int_{\R^{n}}\phi^{2}_h(x)\kappa(\omega,\omega)\para{\displaystyle\int_{0}^{T}
\rho(x+t\omega)\,dt}\,dx=\kappa(\omega,\omega)\displaystyle\int_{0}^{T}\rho(y+t\omega)\,dt
=0,\,\,\,\,
\forall\,\omega\in \s,\;\forall y\in \mathcal{B}_r.$$
Using the fact that $\kappa(\omega,\omega)\neq 0$, we obtain
 $$\displaystyle\int_{0}^{T}\rho(y+t\omega)\,dt
=0,\,\,\,\,
\forall\,\omega\in \s,\;\forall y\in \mathcal{B}_r.$$
Then we have, that
$$\int_{-T}^{T}\rho(y+t\omega)\,dt=0,\,\,\,\,
\forall\,\omega\in \s,\;\forall y\in \mathcal{B}_r.$$
Using that $T>2r$ and $\mbox{Supp}(\rho)\subset B(0,r/2)$, we can find
 \begin{equation}\label{520}
 \int_{\R}\rho(y+t\omega)\,dt=0,\,\,\,\,
\forall\,\omega\in \s,\;\forall y\in \R^n.
\end{equation}
We now turn our attention to the Fourier transform of $\rho$. Let $\xi\in \omega^\perp$. In light of (\ref{520}), we get
$$\int_{\omega^\perp}\int_{\R}e^{-ix\cdot\xi}\rho(y+t\omega)\,dt\,d\sigma=0,\,\,\,\,
\forall\,\omega\in \s, \;\forall y\in \R^n.$$
where $d\sigma$ is the $(n-1)$-dimensional standard volume on $\omega^{\perp}$.
By the change of variable  $x=y+t\omega\in \omega^{\perp}\oplus\R \omega=\R^{n},\,\,dy=dt\,d\sigma,$  we have

$$\widehat{\rho}(\xi)=\int_{\R^n}e^{-ix\cdot \xi }\rho(x)dx=0,\quad \forall\xi  \in \omega^{\perp},\; \forall \omega\in \s.$$
Hence, by the injectivity of the Fourier transform we get the desired
result.
This ends the proof.
\appendix\label{AA}
\section{}
\setcounter{equation}{0}
\begin{lemma}\label{A1}
Let $p\geqslant 1$ and $(a,k)\in \F$, $v\in L^p(Q_T)$ and $u_0\in L^p(Q)$. We consider the
 following problem
\begin{equation}
\left\{
  \begin{array}{ll}
\partial_t u+\theta\cdot\nabla u+a(t,x) u=\G_{k}[u]+v& \mbox{in}\quad(0,T)\times Q, \\      
 u(0,x,\theta)=u_0& \mbox{in}\quad \s\times\Omega,\\
u(t,x,\theta)=f(t,x,\theta)& \mbox{on}\quad  \Sigma^-,\\
\end{array}
\right.\label{r1}
\end{equation}
where $f\in \mathscr{L}^-_p(\Sigma^-)$. Then, there is a constant $C>0$ such that
\begin{equation}\label{1.599}
\norm{u(t,\cdot,\cdot)}_{L^p(Q)}+\Vert u \Vert _{\mathscr{L}^{+}_p(\Sigma^{+})}\leq C(\norm{u_0}_{L^p(Q)}+ \Vert f \Vert _{\mathscr{L}^{-}_p(\Sigma^{-})}+\norm{v}_{L^p(Q_T)}),\quad \forall\,t\in (0,T).
\end{equation}
\end{lemma}
\begin{proof}
Multiplying  the first equation of \eqref{r1} by the complex conjugate of  $\vert u\vert^{p-2}\overline{u}$,
integrating it over $Q$ and taking its real part, we get:  
\begin{multline*}
 \mbox{Re}  \int_Q (\partial_t u(t) \vert u(t)\vert^{p-2}\overline{u}(t)+ \theta \cdot \nabla  u(t) \vert u (t)\vert^{p-2}\overline{u}(t) + a(t,x) u(t) \vert u(t)\vert^{p-2} \overline{u} (t)) d\theta dx \\
 =\mbox{Re} \int_Q \G_{k}[u](t)\vert u(t) \vert^{p-2}\overline{u}(t)d\theta dx+\mbox{Re} \int_Q v(t)\vert u(t) \vert^{p-2}\overline{u}(t)d\theta dx.
\end{multline*}
Then, we have
\begin{multline*}
 \frac{1}{p} \frac{d}{dt}\int_Q \vert u(t)\vert^p d\theta dx +\frac{1}{p}\int_Q \textrm{div}(\theta \vert u(t)\vert^p)d\theta dx  + \int_Q a(t,x) \vert u(t)\vert^p  d\theta dx \\
 =\mbox{Re} \int_Q \G_{k}[u](t)\vert u (t)\vert^{p-2}\overline{u}(t) d\theta dx+\mbox{Re} \int_Q v(t)\vert u(t) \vert^{p-2}\overline{u}(t)d\theta dx. 
\end{multline*}
Then by Gauss identity, we get
 \begin{multline}
\frac{1}{p} \frac{d}{dt}\int_Q \vert u(t)\vert^p d\theta dx +\frac{1}{p} \int_{\Gamma^+}\theta\cdot \nu(x) \vert u(t)\vert^p d\theta dx +\frac{1}{p} \int_{\Gamma^-}\theta\cdot \nu(x) \vert u(t)\vert^p d\theta dx  \\+ \int_Q a(t,x) \vert u(t)\vert^p  d\theta dx =\mbox{Re} \int_Q (\G_{k}[u](t)\vert u(t) \vert^{p-2}\overline{u}(t) d\theta dx+\mbox{Re} \int_Q v(t)\vert u(t) \vert^{p-2}\overline{u}(t)d\theta dx.\label{r11} 
 \end{multline} 
By  (\ref{2.7}) and Hölder's inequality, we get
\begin{equation}
\begin{split}
\int_Q \vert \G_{k}[u(t)]\vert \vert u(t)\vert^{p-1}dx d\theta &\leq \left(\int_Q  \vert  \G_{k}[u(t)]\vert^p dx d\theta \right)^{\frac{1}{p}} \left( \int_Q  \vert u(t)\vert^{(p-1)q} dx d\theta \right)^{\frac{1}{q}} \\
&\leq C_p \Vert u(t)\Vert_{L^p(Q)} \left( \int_Q \vert u(t)\vert^p dx d\theta \right)^{1-\frac{1}{p}}\\ 
&\leq C_p \Vert u(t)\Vert_{L^p(Q)} \Vert u(t)\Vert^{p(1-\frac{1}{p})}_{L^p(Q)}\\
&\leq C_p \Vert u(t)\Vert^p_{L^p(Q)},\label{r2}
\end{split}
\end{equation}
where $C_p=M_1^{1/q}M_2^{1/p}$ with $\frac{1}{p}+\frac{1}{q}=1.$
 On the other hand Young inequality implies
\begin{equation}\label{..}
 \int_Q \vert v(t)\vert \vert u(t) \vert^{p-1} d\theta dx\leq \frac{1}{p}\norm{v(t)}^p_{L^p(Q)}+\frac{1}{q}\norm{u(t)}^p_{L^p(Q)}.
\end{equation}
Then, we obtain
\begin{equation}\label{a1}
\begin{split}
\frac{1}{p} \frac{d}{dt}\int_Q \vert u(t)\vert^p d\theta dx+\frac{1}{p} \int_{\Gamma^+}\vert\theta\cdot \nu(x)\vert &\vert u(t)\vert^p d\theta dx  \leq \frac{1}{p} \int_{\Gamma^-}\vert\theta \cdot \nu(x)\vert \vert f(t)\vert^pd\theta dx+C_p \Vert u(t)\Vert^p_{L^p(Q)}\\&+\Vert a\Vert_{L^\infty(Q)}\Vert u(t)\Vert^p_{L^p(Q)}+ \frac{1}{p}\norm{v(t)}^p_{L^p(Q)}+\frac{1}{q}\norm{u(t)}^p_{L^p(Q)},
\end{split}
\end{equation}
we conclude
$$\frac{d}{dt} \Vert u(t) \Vert^p_{L^p(Q)}\leq \int_{\Gamma^-}\vert\theta \cdot \nu(x)\vert \vert f(t)\vert^pd\theta dx+  C_1\Vert u(t) \Vert^p_{L^p(Q)}+\norm{v(t)}^p_{L^p(Q)}.$$
Now, integrating this last inequality over $(0,t)$, we get
\begin{equation}
 \Vert u(t)\Vert^p_{L^p(Q)}\leq  \Vert f \Vert^p_{\mathscr{L}^{-}_p(\Sigma^{-})} +\norm{v}^p_{L^p(Q_T)}+ \Vert u_0\Vert^p_{L^p(Q)}+  C_1\int_0^t\Vert u(s) \Vert^p_{L^p(Q)}ds.\label{r30}
 \end{equation}
By using Grönwall's inequality, we get
 \begin{equation}
  \Vert u(t)\Vert^p_{L^p(Q)}\leq C_T \left( \Vert f \Vert^p_{\mathscr{L}^{-}_p(\Sigma^{-})} +\norm{v}^p_{L^p(Q_T)}+ \Vert u_0\Vert^p_{L^p(Q)}\right) .\label{r3}
 \end{equation}
By integrating \eqref{a1} over $(0,T)$, we get
 \begin{equation}
\frac{1}{p} \Vert u(T) \Vert^p_{L^p(Q)}-\frac{1}{p}\Vert u_0 \Vert^p_{L^p(Q)}+\frac{1}{p} \int_{0}^T\Vert u(s) \Vert^p_{L^p(\Gamma^{+};d\xi)}ds \leq \frac{1}{p} \int_0^T\Vert f(s) \Vert^p_{L^p(\Gamma^{-};d\xi)}ds+C_1 \int_0^T\Vert u(s)\Vert^p_{L^p(Q)}ds.
\end{equation}
Therefore by using \eqref{r3}, we get
 \begin{equation}\label{a2}
  \Vert u \Vert^p_{\mathscr{L}^{+}_p(\Sigma^{+})}  \leq C\left( \Vert f \Vert^p_{\mathscr{L}^{-}_p(\Sigma^{-})} +\norm{v}^p_{L^p(Q_T)}+ \Vert u_0\Vert^p_{L^p(Q)}\right). 
\end{equation}
In light of \eqref{r3} and \eqref{a2}, the proof of Lemma \eqref{A1} is completed.
\end{proof}



\end{document}